\documentclass[12pt]{amsart}
\usepackage{a4wide}
\usepackage{amssymb,amsfonts,amsmath}

\newtheorem{thm}{Theorem}[section]
\newtheorem{cor}[thm]{Corollary}
\newtheorem{lem}[thm]{Lemma}
\newtheorem{prop}[thm]{Proposition}
\newtheorem{defn}[thm]{Definition}
\newtheorem{rem}[thm]{Remark}

\usepackage{latexsym}
\usepackage{leftidx}

\newcommand{\setR}{{\mathord{\mathbb R}}}

\newcommand{\mymarginpar}[1]{}
\newcommand{\M}{{\mathord{\mathcal M}}}

\def\supp{\mathop{\mathrm{supp}}}

\def\part{\mathop{\partial}}

\usepackage[colorlinks]{hyperref}

\numberwithin{equation}{section}

\usepackage{comment}

\usepackage{framed}

\begin{document}
\title[ Elliptic equations in Weighted Besov spaces ]
 {Elliptic equations in Weighted Besov spaces  on asymptotically flat Riemannian
manifolds}
\author[U. Brauer]{Uwe Brauer${}^\dagger$}\thanks{$\dagger$
 Uwe Brauer's research was partially supported by grant MTM2012-31928}

\address{%
  Departamento de  Matem\'atica Aplicada\\ Universidad Complutense Madrid
28040 Madrid, Spain}
\email{oub@mat.ucm.es}
    
\author[L. Karp]{Lavi Karp*}

\address{%
Department of Mathematics\\ ORT Braude College\\
P.O. Box 78, 21982 Karmiel\\ Israel}

\email{karp@braude.ac.il}

\thanks{*Research  supported  ORT
Braude College's Research Authority}



\begin{abstract}
  This paper deals with the applications of weighted Besov spaces to
  elliptic equations on asymptotically flat Riemannian manifolds, and
  in particular to the solutions of Einstein's constraints equations.
  We establish existence theorems for the Hamiltonian  and the
  momentum constraints with constant mean curvature and with a
  background metric that satisfies very low regularity assumptions.

  These results extend the regularity results of Holst, Nagy and
  Tsogtgerel about the constraint equations on compact manifolds in
  the Besov space $B_{p,p}^s$ \cite{Holst_Nagy_Tsogtgerel}, to
  asymptotically flat manifolds.
  We also consider the Brill--Cantor criterion in the weighted Besov
  spaces.
  Our results improve the regularity assumptions on asymptotically
  flat manifolds \cite{Choquet_Isenberg_Pollack,
    maxwell05:_solut_einst}, as well as they enable us to construct
  the initial data for the Einstein--Euler system.

\end{abstract}

\maketitle

\section{Introduction}
\label{sec:introduction}

A special feature of the Einstein equations is that initial data
cannot be prescribed freely.
They must satisfy constraint equations.
To prove the existence of a solution of the Einstein equations, it is
first necessary to prove the existence of a solution of the
constraints.
The usual method of solving the constraints relies on the theory of
elliptic equations.
Since asymptotically flat metric falls off only as $r^{-1}$ as
$r\to\infty$ on a spacelike slice and the positive mass theorem
\cite{MR526976
}implies that any attempt to impose faster fall-off excludes all but
the trivial solution the metric does not belong to a Sobolev space.
The usual way to get around this is to replace the ordinary Sobolev
space by a weighted one.

Therefore much attention has been devoted to solutions of the Einstein
constraint equations in asymptotically flat space--times by means of
weighted Sobolev spaces as an essential tool.
These spaces are defined as the completion of $C_0^\infty(\mathbb{R}^n)$
with respect to the norm
\begin{equation}
 \label{eq:intro:5}
\|u\|_{{m,p,\delta}}
 =\left(\sum_{|\alpha|\leq m}\int\left|(1+|x|)^{\delta+|\alpha|}
\nobreakspace\nobreakspace\partial^\alpha u\right|^p dx\right)^{\frac{1}{p}},
\end{equation}
and denoted by $W_{m,\delta}^p$.

Elliptic equations on $W_{m,\delta}^p$ spaces were first considered by
Nirenberg and Walker in
\cite{nirenberg73:_null_spaces_ellip_differ_operat_r}.
This paper led to numerous publications dealing with its applications
to the solutions of Einstein constraint equations in asymptotically
flat space--times.
Some significant contributions include the papers of Bartnik
\cite{bartnik86}, Cantor \cite{cantor75:_spaces_funct_condit_r,
  cantor79}, Choquet--Bruhat and Christodoulou
\cite{choquet--bruhat81:_ellip_system_h_spaces_manif_euclid_infin} and
Christodoulou and O'Murchadha \cite{OMC}.
Afterward the regularity assumptions were improved, by
Choquet--Bruhat, Isenberg, Pollack and York for the Einstein--scalar
field gravitational constraint equations
\cite{Choquet_Isenberg_Pollack, y.00:_einst_euclid}, and by Maxwell in
the vacuum case and with boundary conditions
\cite{maxwell05:_solut_einst}.
In both papers the authors assumed that the background metric is
locally in $W_2^p$ when $p>\frac{n}{2}$, and they obtained a solution
to the constraint equations with a conformal metric in the same space.
In the case $p=2$, Maxwell constructed low regularity solutions of the
vacuum Einstein constraint equations with a metric in the weighted
Bessel
potential spaces $H^s_{\rm loc}$ for   $s>\frac{n}{2}$
\cite{maxwell06:_rough_einst}.

On compact manifolds one seeks solutions to the Einstein constraint
equations in the unweighted Sobolev spaces.
Choquet--Bruhat obtained solutions with a metric in $W_{2}^p$ and for
$p>\frac{n}{2}$ \cite{choquet_04}.
Later Maxwell improved the regularity in the Bessel potential spaces
$H^s$ for $s>\frac{n}{2}$ \cite{maxwell05:_rough_einst}.
Holst, Nagy and Tsogtgerel \cite{Holst_Nagy_Tsogtgerel} study
solutions of the Einstein constraints in the Sobolev--Sobolevskij
spaces $W_s^p$, and obtained solutions to the Hamiltonian and momentum
constraints in these spaces when $s\in (\frac{n}{p},\infty)\cap
[1,\infty)$ and $p\in(1,\infty)$.
Thus their results cover \cite{choquet_04} in the case $s=2$ and
\cite{maxwell05:_rough_einst} in the case $p=2$.

One of our main results is the extension of the regularity result of
\cite{Holst_Nagy_Tsogtgerel} to asymptotically flat manifolds.
In doing so we use Triebel's extension of the
$W_{m,\delta}^p$--spaces to the fractional order spaces
$W_{s,\delta}^p$--spaces, where $s$ and $\delta$ are real numbers,
\cite{triebel76:_spaces_kudrj2} (see definition \ref{def:spaces}).
The $W_{s,\delta}^p$--spaces are constructed by means of the Besov
space $B_{p,p}^s$ (see (\ref{eq:norm2})).
The Besov spaces coincide with the Sobolev--Sobolevskij spaces whenever $s$ is
positive (see e.g. \cite[Ch.6]{berg_lofstrom_76}) , and they are suitable for
interpolation both for negative and positive $s$.
This is a vital property of the Besov spaces that enable us to prove
certain properties by means of interpolation.

In the present paper we prove existence and uniqueness results for the
Hamiltonian and momentum constraints with constant mean curvature
(CMC) in the $W_{s,\delta}^p$--spaces, and under the conditions $s\in
(\frac{n}{p},\infty)\cap [1,\infty)$, $\delta\in
(-\frac{n}{p},n-2-\frac{n}{p})$ and for all $p\in(1,\infty)$.

Thus we improve the regularity \cite{Choquet_2009,
  Choquet_Isenberg_Pollack, y.00:_einst_euclid,
  maxwell05:_solut_einst} and extend the range of $p$  to $(1,\infty)$.
\cite{Choquet_Isenberg_Pollack, maxwell05:_solut_einst}.
Holst et al.~achieved solutions to the
constrain equations on compact manifolds  with low
regularity,  without the CMC--condition.
We believe that the present paper is a step toward developing
a theory of rough solutions in asymptotically flat manifolds for the
non--CMC case.

The Brill--Cantor condition (\ref{eq:brill-cantor}) suggests a criterion under 
which a given metric in an asymptotically flat manifold can be rescaled to
yield 
a conformal metric with zero scalar curvature (see 
\S\ref{sec:brill-cantor-theorem}). This  criterion is related to the Yamabe 
conformal invariant classes, but has a different interpretation on 
asymptotically flat manifolds.   For an enlightening  discussion about this 
criterion see \cite{friedrich_11}. Cantor and Brill suggested this condition
and showed its equivalent to the existence of  a flat metric
for an integer $m$ greater than  $\frac{n}{p}+2$ and
$1<p<\frac{2n}{n-2}$, \cite{cantor81:_laplac}.
Since then  the regularity assumptions  were improved by several authors
\cite{Choquet_2009, 
Choquet_Isenberg_Pollack, y.00:_einst_euclid, maxwell05:_solut_einst}, however, 
they dealt only with Sobolev spaces of integer order, and under the
restriction that  $p>\frac{n}{2}$. For $p=2$ it was  proved for all
$s>\frac{n}{2}$ 
in \cite{maxwell06:_rough_einst}.

We treat the Bill--Cantor condition in the weighted Besov spaces
$W_{s,\delta}^p$ and establish its equivalent to the
  existence of  a metric with zero scalar curvature for $s\in
(\frac{n}{p},\infty)\cap [1,\infty)$, $\delta\in
(-\frac{n}{p},n-2-\frac{n}{p})$ and all $p\in(1,\infty)$.
To conclude, our results generalize \cite{maxwell06:_rough_einst} to
$p\in(1,\infty)$, and improve the regularity of \cite{Choquet_2009,
  Choquet_Isenberg_Pollack, y.00:_einst_euclid,
  maxwell05:_solut_einst}.

One of the essential difficulties is to prove the continuity of the Yamabe
functional (\ref{eq:brill-cantor}) in terms of the norm of $W_{s,\delta}^p$. In 
previous publications  the restriction of $p>\frac{n}{2}$  was caused  by two
reasons: the Sobolev embedding theorem; and  a certain application of the
generalized H\"older inequality that requires the condition $p>\frac{n}{2}$. We,
however,  overcome this obstacle by improving the multiplication property of
functions in $W_{s,\delta}^p$--spaces (see Proposition \ref{prop:11}).

The outline of the paper is as follows.  In the first part of Section
\ref{sec:Besov} we sketch Triebel's construction of the weighted Besov spaces
and state their main properties. We also pay attention to the bilinear form on
these  spaces  since the dual representation  of the norm
plays a role in the proof of
existence of solutions of nonlinear equations.
In the second part we establish tools which
are needed for PDE in these spaces, including embeddings, pointwise
multiplication and Moser type estimates.

Section \ref{sec:linear_asymp} is devoted to elliptic linear systems on
asymptotically flat Riemannian manifolds. In the first subsection we establish
\textit{ a priori} estimates for second order elliptic operators with
coefficients in the $W_{s,\delta}^p$ spaces in $\setR^n$  and  show
that these systems are
semi--Fredholm operators. This property plays
an essential role in the study of
the non--linear equations. The definition of asymptotically flat
manifolds of the class
$W_{s,\delta}^p$ is done in subsection \ref{sec:Asymptotically flat manifold}. 
We also define there the norms and the 
bilinear forms on a Reimannian manifold.
In subsection \ref{sec:weak} we study weak solutions that meet very
low regularity requirements. This  demands special
attention to the
extension of a $L^2$-bilinear form  to the bilinear form acting on
$W_{s,\delta}^p$ and its dual. We then define weak solutions on the manifolds
and derive a weak maximum principle for all $p\in (1,\infty)$.

In Section \ref{sec:semi-linear}  we prove the existence and uniqueness theorem
of a semi--linear
equation, where the linear part is the Laplace--Beltrami operator of an
asymptotically flat Riemannian manifold. The method of sub and super solution
is the common method for these types of non--linearity, however, we shall
implement  Cantor's homotopy argument \cite{cantor79} in the weighted Besov 
spaces.

In Section \ref{sec:brill-cantor-theorem} 
we discuss  the Brill--Cantor criterion.
We show that for $s\in
(\frac{n}{p},\infty)\cap [1,\infty)$, $\delta\in
(-\frac{n}{p},-2-n+\frac{n}{p})$ and all $p\in(1,\infty)$ condition 
(\ref{eq:brill-cantor}) is necessary and sufficient for
the existence of a metric that belongs to $W_{s,\delta}^p$ and
has zero scalar curvature. In Section
\ref{sec:appl-constr-equat}  we consider
the construction of initial data
for the Einstein--Euler system. In this system the equations for the
gravitational fields are coupled to a perfect fluid and certain relations
between the source terms of the
constraint equations  and the fluid variables must be fulfilled.
In \cite{BK3} the authors discussed this problem in
detail  using
the weighted Hilbert spaces $W_{s,\delta}^2$. Here we extend these
results to the $W_{s,\delta}^p$ spaces.

Finally, in the Appendix we discuss the extension of several properties 
$W_{s,\delta}^p$ spaces in $\setR^n$ to the $W_{s,\delta}^p$ spaces on
a Riemannian manifold. 
The norm on a Riemannian manifold is defined by means of a collection of charts 
and partition of unity. Though the extension to a Riemannian manifold seems to 
be an obvious matter,  sometimes it requires a certain attention and 
watchfulness. 

\textbf{Some notations:} For $p\in(1,\infty)$, $p'$ will stand for the
dual index to $p$, 
that is $\frac{1}{p}+\frac{1}{p'}=1$. The scaling of a distribution $u$ with a
positive number $\varepsilon$ is denoted by $u_\varepsilon$. A Riemannian
manifold is denoted by $\M$ and  $g=g_{ab}$ is a metric on $\M$, $\nabla u$ is
the covariant derivative and $|\nabla u|_g^2=g^{ab}\partial_a
u\partial_b u$,
where $g^{ab}$ is the inverse matrix of $g_{ab}$. Latin indexes $a,b$ take the
values $1,\ldots,n$ and the dimension $n$ is greater or equal to two throughout
this paper. We will use the
notation $A\lesssim B$ to denote an inequality $A\leq CB$ where the positive
constant $C$ does not depend on the parameters in question.

\subsection*{Acknowledgments}
We would thanks to the anonymous  referee for his/her constructive comments, 
which definitively helped to improve the manuscript.

\section{The weighted Besov spaces}
\label{sec:Besov}
\subsection{The construction of the Spaces $W_{s,\delta}^p$}
\label{sec:constr-spac-h_s}

In this subsection we sketch the construction of the weighted Besov
spaces.
We start by fixing the notations and recalling the
definition of the Besov spaces $B_{p,p}^s$
\cite{berg_lofstrom_76,
  triebel83}.
Let $\mathcal{S}$
denote the Schwartz class of rapidly decreasing functions in $\setR^n$ and
$\mathcal{S}'$  its dual. Let
$\{\phi_j\}_{j=0}^\infty\subset C_0^\infty(\setR^n)$
be a  dyadic partition of
unity of $\setR^n$ such that ${\rm supp}(\phi_0)\subset \{|\xi|\leq
2\}$, ${\rm supp}(\phi_j)\subset \{2^{j-1}\leq|\xi|\leq 2^{j+1}\}$ for 
$j\in\mathbb{N}$,  all the
function $\phi_j$ are nonnegative and
 $\sum_{j=0}^\infty \phi_j(\xi)=1$.
 Let
$\mathcal{F}(u)$ be the Fourier transform of a distribution 
$u,$ $s\in\setR$ and $1\leq p<\infty$, then
\begin{equation}
\label{eq:norm2}
 W^p_s:=B_{p,p}^s=\left\{u\in \mathcal{S}':
\|u\|_{W_s^p}:=\left(\sum_{j=0}^\infty
2^{jsp}\left\|\mathcal{F}^{-1}(\phi_j \mathcal{F}(u))\right\|_{L^p}^p
\right)^{1/p}<\infty\right\}.
\end{equation}

For $1<p<\infty$, the dual space $\left(W_s^p\right)'$ is isomorphic to
$W_{-s}^{p'}$, where 
$1/p+1/p'=1$ (see e.g. \cite[Corollary 6.2.8]{berg_lofstrom_76}).   Let $u\in 
W_s^p$ and 
$\varphi\in W_{-s}^{p'}$, then
\begin{equation}
\label{eq:3}
\begin{split}
 \langle u,\varphi\rangle &=\sum_{j=0}^\infty\sum_{k=j-2}^{k=j+2}
 \int 
\left[\mathcal{F}^{-1}(\phi_j 
\mathcal{F}(u))(x)\right]\left[\mathcal{F}^{-1}(\phi_k 
\mathcal{F}(\varphi))(x)\right]dx  \\ &:=
\sum_{j=0}^\infty\sum_{k=j-2}^{j+2}\left(\mathcal{F}^{-1}(\phi_j 
\mathcal{F}(u)),\mathcal{F}^{-1}(\phi_k 
\mathcal{F}( \varphi))\right)_{L^2(\setR^n)},
\end{split}
\end{equation} 
is a bilinear form on $W_s^p\times W_{-s}^{p'}$.
Here and throughout the paper $(\cdot,\cdot)_{L^2}$ denote the $L^2$ bilinear 
form, also whenever $k<0$,  sums as in (\ref{eq:3}) start from zero. 
For the proof of this formula see \cite[Proposition 2.76]{Bahouri_2011}.  
\begin{prop}
\label{prop:dual}
 If  $s>0$ $u\in W_s^p$ and $\varphi\in \mathcal{S}$, then 
\begin{equation}
\label{eq:form:1}
 \langle u,\varphi\rangle=(u,\varphi)_{L^2(\setR^2)}.
\end{equation} 

\end{prop}
 \begin{proof}
Assuming that $u$ is smooth, then we have by  Parseval's formula that
 \begin{equation*}
  \begin{split}  
& (u,\varphi)_{L^2(\setR^n)}=c_n\left(\mathcal{F}(u),\overline{\mathcal{F}
(\varphi)}\right)_ { 
L^2(\setR^n)}=c_n\left(\sum_{j=0}^\infty 
\phi_j\mathcal{F}(u),\overline{\sum_{k=0}^\infty 
\phi_k\mathcal{F} (\varphi) } \right)_ {
L^2(\setR^n)} \\ = &c_n\sum_{j=0}^\infty\sum_{k=j-2}^{j+2}\left(
\phi_j\mathcal{F}(u),\overline{ 
\phi_k\mathcal{F} (\varphi) } \right)_ {
L^2(\setR^n)}=\sum_{j=0}^\infty\sum_{k=j-2}^{j+2}\left(\mathcal{F}^{-1}
(\phi_j 
\mathcal{F}(u)),\mathcal{F}^{-1}(\phi_k 
\mathcal{F}( \varphi))\right)_{L^2(\setR^n)} \\
= &\langle u,\varphi\rangle.
  \end{split}
 \end{equation*}
Now if  $u\in W_{s}^p$, then $u\in L^p$ since $s>0$. So take a sequence 
$\{u_k\}\subset
C_0^\infty(\setR^n)$ such that $u_k$ tends to $u$ in $W_s^p$.  Then 
$|(u_k,\varphi)_{L^2(\setR^n)}|\leq \|u_k\|_{L^p}\|\varphi\|_{L^{p'}}\leq  
\|u_k\|_{W^p_s}\|\varphi\|_{L^{p'}}$. Thus (\ref{eq:form:1}) follows.
\end{proof}

We turn now to  the construction of the weighted-$W_s^p$ space. We also use a
dyadic partition  of unity,  which is denoted by
$\{\psi_j\}_{j=0}^\infty$, and 
is such that the support  of $\psi_j$ is contained in the dyadic
shell $\{x: 2^{j-2}\leq |x|\leq 2^{j+1}\}$, $\psi_j(x)=1$ on $\{x: 2^{j-1}\leq
|x| \leq 2^{j}\}$ for $j=1,2,...$, while $\psi_0$ has a support in the ball
$\{x: |x|\leq 2\}$ and $\psi_0(x)=1$ on $\{x: |x|\leq 1\}$. In addition we
require that $\{\psi_j\}_{j=0}^\infty\subset C_0^\infty(\setR^n)$ and satisfies
the inequalities
\begin{equation}
  \label{eq:const:4}
 |\partial^\alpha  \psi_j(x)|\leq  C_\alpha  2^{-|\alpha|j},
\end{equation}
where the constant $C_\alpha$ does not depend on $j$. For a positive number
$\varepsilon$ we denote the scaling $u(\varepsilon x)$ by $u_\varepsilon(x)$.
\begin{defn}[Weighted Besov  spaces $W_{s,\delta}^p$]
\label{def:spaces}
 Let $s,\delta\in\setR$ and $p\in[1,\infty)$, the
$W_{s,\delta}^p(\setR^n)$-space  is
the set of all tempered distributions $u$ such that the norm
\begin{equation}
\label{eq:norm1}
 \|u\|_{W_{s,\delta}^p(\setR^n)}^p:=\sum_{j=0}^\infty 2^{(\delta+\frac{n}{p})pj}
\left\|\left(\psi_j u\right)_{(2^j)}\right\|_{W_s^p}^p.
\end{equation}   
is finite.
\end{defn}
 The $W_{s,\delta}^p$-norm of distributions in an
open set $\Omega\subset\setR^n$ is given by
\begin{equation*}
\label{eq:norm4}
 \left\|u \right\|_{W_{s,\delta}^p(\Omega)}=\inf\limits_{f{_{\mid_\Omega}}=u}
\left\|f \right\|_{W_{s,\delta}^p(\setR^n)}.
\end{equation*}

The following basic properties were established in Triebel 
\cite{triebel76:_spaces_kudrj2, triebel76:_spaces_kudrj}.
\begin{thm}[Triebel, Basic properties]
\label{thm1}
Let $s,\delta\in \setR$ and $p\in(1,\infty)$.
 \label{thm:Triebel}
\begin{enumerate}
 \item[{\rm (a)}] The space $W_{s,\delta}^p(\setR^n)$ is a Banach space and 
 different  choices  of  the dyadic resolution $\{\psi_j\}$ which satisfies
(\ref{eq:const:4}) result in equivalent norms. 

\item[{\rm (b)}]  $C_0^\infty(\setR^n)$ is a dense subset in
$W_{s,\delta}^p(\setR^n)$.
\item[{\rm (c)}]  The dual space of  $W_{s,\delta}^p(\setR^n)$ is
$W_{-s,-\delta}^{p'}(\setR^n)$, where $p'=\frac{p}{(p-1)}$.

\item[{\rm (d)}] 
\label{thm:interpolation}
Interpolation (real): Let $0<\theta<1$, $s=\theta
s_0+(1-\theta)s_1$, $\delta=\theta \delta_0+(1-\theta)\delta_1$ and $1/p=\theta
/p_0+(1-\theta)/p_1$, then 
\begin{equation*}
 \left(W_{s_1,\delta_1}^{p_1}(\setR^n),
W_{s_2,\delta_2}^{p_2}(\setR^n)\right)_{\theta,p}=W_{s,\delta}^p(\setR^n).
\end{equation*} 
\end{enumerate}
\end{thm}

For the definition of the   bilinear form on $W_{s,\delta}^p(\setR^n)\times
W_{-s,-\delta}^{p'}(\setR^n)$ we choose  a dyadic resolution $\{\psi_j\}$ such
that $\sum_{j=0}^\infty \psi_j(x)=1$ and set 
\begin{equation}
 \label{eq:4}
\langle u,\varphi \rangle_W=\sum_{j=0}^\infty\sum_{k=j-2}^{j+2}2^{nj}
\langle (\psi_j u)_{(2^j)}, (\psi_k\varphi)_{(2^j)}\rangle,
\end{equation} 
$\langle \cdot,\cdot\rangle$ denotes the form appearing on the left
hand side and is
of course given by (\ref{eq:3}).  It satisfies the inequality 
\begin{equation}
\label{eq:norm:8}
 |\langle u,\varphi \rangle_W|\leq C
\|u\|_{W_{s,\delta}^p(\setR^n)}\|\varphi\|_{W_{-s,-\delta}^{p'}(\setR^n)} 
\end{equation}
(see the proof of Theorem 2 in \cite{triebel76:_spaces_kudrj}).
In a similar manner to Proposition \ref{prop:dual}, we have that:
\begin{prop}
 \label{prop:dual2}
If $s>0$, $u\in W_{s,\delta}^p(\setR^n)$  and $\varphi\in \mathcal{S}$, then
\begin{equation}
\label{eq:dual2}
 (u,\varphi)_{L^2(\setR^n)}=\langle u, \varphi\rangle_W.
\end{equation}
\end{prop}
\begin{proof}
 We use Proposition  \ref{prop:dual} and also the same idea as in its proof,
but here  the dyadic resolution  $\{\psi_j\}$ replaces $\{\phi_j\}$.
\end{proof}
\begin{rem}
\label{rem:3}
Let $f\in W_{s,\delta}^p(\setR^n)$, then
it follows from  (c)  above that 
\begin{equation}
\label{eq:dual}
 \|f\|_{W_{s,\delta}^{p}(\setR^n)}=\sup\{|\langle f,\varphi\rangle_W|:
\|\varphi\|_{W_{-s,-\delta}^{p'}(\setR^n)}\leq 1, \varphi\in
C_0^\infty(\setR^n)\},
\end{equation}
and in particular, by (c) of Theorem \ref{thm1} and Proposition
\ref{prop:dual2},
if $f\geq 0$, then
\begin{equation}
\label{eq:dual3}
 \|f\|_{W_{s,\delta}^{p}(\setR^n)}=\sup\{\langle f,\varphi\rangle_W:
\varphi\geq0, 
\|\varphi\|_{W_{-s,-\delta}^{p'}(\setR^n)}\leq 1, \varphi\in
C_0^\infty(\setR^n)\}.
\end{equation} 
\end{rem}

For $s>0$ the Besov norm (\ref{eq:norm2}) is equivalent 
to the norm
of the Sobolev--Sobolevskij spaces (see e.g.
\cite[Ch.
6]{berg_lofstrom_76}, \cite[\S35]{Tartar} or \cite{triebel83}). Their
norm is defined as follows.  Let $s=m+\lambda$, where $m$
is a nonnegative integer and
$0<\lambda<1$, then
\begin{equation*}    
\label{eq:5}
\displaystyle 
\|u\|_{s,p}^p=\left\{ \begin{array}{ll}\sum_{|\alpha|\leq m}\|\partial^\alpha
u\|_{L^p}^p, & s=m\\ 
\displaystyle \sum_{|\alpha|\leq m}\|\partial^\alpha u\|_{L^p}^p +       
\sum_{|\alpha|= m} \iint 
\frac{ |\partial^\alpha
u(x)-\partial^\alpha u(y)|^p}{ |x-y|^{n+\lambda p}}dxdy, &s=m+\lambda       
\end{array}\right..
\end{equation*}

Thus a natural extension of the  spaces defined  by the norm
(\ref{eq:intro:5}) to the spaces of fractional order, when 
$s>0$, is
\begin{equation}
\label{eq:norm3}
\displaystyle
\|u\|_{s,p,\delta}^p=\left\{ 
\begin{array}{ll}\ 
\displaystyle 
\sum_{|\alpha|\leq
m}\| (1+|x|)^{\delta+|\alpha|} \partial^\alpha
u\|_{L^p}^p, & s=m\\ 
\begin{array}{l}
\displaystyle 
\sum_{|\alpha|\leq m}\| (1+|x|)^{\delta+|\alpha|} \partial^\alpha
u\|_{L^p}^p  + \\      
\displaystyle
\sum_{|\alpha|= m}\iint \frac{ |(1+|x|)^{\delta+m+\lambda} \partial^\alpha
u(x)-(1+ |y|)^{\delta+m+\lambda}\partial^\alpha u(y)|^p}{
|x-y|^{n+\lambda p}}dxdy,
\end{array} &s=m+\lambda           
   \end{array}\right..
\end{equation}

In order to show the equivalence between the norms 
(\ref{eq:norm1}) and (\ref{eq:norm3}) we introduce the  homogeneous
norm, that is,  
\begin{equation*}
\displaystyle
 \|u\|_{s,p,hom}^p=\left\{ \begin{array}{ll}\sum_{|\alpha|=
m}\|\partial^\alpha
u\|_{L^p}^p, & s=m\\
\displaystyle
\sum_{|\alpha|= m}\iint \dfrac{ |\partial^\alpha
u(x)-\partial^\alpha u(y)|^p}{ |x-y|^{n+\lambda p}}dxdy, &s=m+\lambda
   \end{array}\right..
\end{equation*}
Taibleson \cite[Theorem 10]{taibleson64:_lipsc_euclid_space}, proved that the
norm $\|u\|_{s,p}$ is equivalent to
$\left(\|u\|_{L^p}^p+\|u\|_{s,p,hom}^p\right)^{1/p}$ (see also \cite[Ch.
V, \S4-5]{stein70:_singul_integ_differ_proper_funct}), and by using this
equivalence 
Triebel \protect \cite{triebel76:_spaces_kudrj2}
proved that
\begin{equation}
\label{eq:norm5}
c_0 \|u\|_{s,p,\delta}^p \leq\sum_{j=0}^{\infty} 2^{\delta p j}\|\psi_j
u\|_{L^p}^p
+2^{(\delta+ s) p j}\|\psi_j u\|_{s,p,hom}^p\leq c_1 \|u\|_{s,p,\delta}^p.
\end{equation}
Moreover, he showed that the constants in  the above equivalence
depend only on $s$, $\delta$, $p$,
the dimension and the constants $C_\alpha$ of inequalities (\ref{eq:const:4}). 
   
Taking into account the homogeneous properties, that is, $\|(\psi_j
u)_{2^j}\|_{L^p}^p=2^{-jn}\|\psi_j u\|_{L^p}^p$ and
$\|(\psi_ju)_{2^j}\|_{s,p,hom}^p=2^{-j(n-sp)}\|\psi_ju\|_{s,p,hom}^p$, and
combining   them  with the  equivalence (\ref{eq:norm5}), we obtain 
\begin{equation*}
\begin{split}
 \|u\|_{s,p,\delta}^p  & \sim\sum_{j=0}^{\infty} 2^{(\delta +\frac{n}{p})p
j}\left(\|(\psi_j u)_{2^j}\|_{L^p}^p +\|(\psi_j u)_{2^j}\|_{s,p,hom}^p\right)
\\ &\sim
 \sum_{j=0}^{\infty} 2^{(\delta+\frac{n}{p})p
j}\|(\psi_j u)_{2^j}\|_{s,p}^p
\sim
 \sum_{j=0}^{\infty} 2^{(\delta+\frac{n}{p})p
j}\|(\psi_j u)_{2^j}\|_{> W_{s}^p}^p=
\|u\|_{W_{s,\delta}^p(\setR^n)}^p.
\end{split}
\end{equation*}
This proves the following theorem of Triebel \cite{triebel76:_spaces_kudrj2}.
\begin{thm}[Triebel, Equivalence of norms]
\label{thm:2}
Let $s>0$, $1\leq p<\infty$ and
$-\infty<\delta<\infty$. Then the norms
(\ref{eq:norm1}) and (\ref{eq:norm3}) are
equivalent. In particular, when $s$ is a non--negative integer, then the norm
(\ref{eq:norm1}) is equivalent  to the norm  (\ref{eq:intro:5}).
\end{thm}
\begin{rem}
\label{rem:2.7}
Note that the Besov space  $B_{p,p}^0$ is continuously included in $L^p$
for $p\in [1,2]$,
and  $L^p$  is continuously included in $B_{p,p}^0$ for $p\in
[2,\infty)$ (see  \cite[Theorem 2.41]{Bahouri_2011}). This phenomenon occurs
also  in the weighted spaces. Let $L_\delta^p$ denote the Lebesgue space with
the weight $(1+|x|)^\delta$, then it follows from the dyadic representation of the
norm that $W_{0,\delta}^p$ is continuously included in $L_\delta^p$ for $p\in
[1,2]$,
and  $L_\delta^p$  is continuously included in $W_{0,\delta}^p$ for $p\in
[2,\infty)$.
\end{rem}

\subsection{Some Properties of $W_{s,\delta}^p(\setR^n)$-spaces}
\label{sec:some-properties-h_s}

In this subsection we establish several useful tools for PDEs
in these spaces, including embeddings, pointwise
multiplications, fractional powers and Moser type estimates.
\begin{prop}
 \label{prop:1}
If $u\in W_{s,\delta}^p(\setR^n)$, then
\begin{equation}
\label{eq:norm7}
 \|\partial_i u\|_{W_{s-1,\delta+1}^p(\setR^n)}\leq C
\|u\|_{W_{s,\delta}^p(\setR^n)},
\end{equation}
where the constant $C$ depends on the constant of the equivalence of Theorem
\ref{thm:2}.
\end{prop}
\begin{proof}
If $s\geq 1$, then (\ref{eq:norm3}) implies $\|\partial_i
u\|_{s-1,p,\delta+1}
\leq \|u\|_{s,p,\delta}$, so (\ref{eq:norm7}) follows from Theorem
\ref{thm:2} in that case. For $s\leq 0$, 
we have by the previous step
\begin{equation*}
|\langle\partial_i u, \varphi\rangle_W|=|\langle u, \partial_i\varphi\rangle_W| 
\leq
\|u\|_{W_{s,\delta}^p(\setR^n)}\|\partial_i
\varphi\|_{W_{-s,-\delta}^{p\prime}(\setR^n)}\leq C
\|u\|_{W_{s,\delta}^p(\setR^n)}\|\varphi\|_{W_{-s+1,-\delta-1}^{p\prime}
(\setR^n) } 
\end{equation*}
for all $\varphi\in C_0^\infty(\setR^n)$.
Hence by (\ref{eq:dual}),
\begin{math}
 \|\partial_i u\|_{W_{s,\delta}^p(\setR^n)}
\leq C \|u\|_{W_{s+1,\delta-1}^p(\setR^n)}.
\end{math}
For the remaining  value of $s$ we use interpolation in order
to  obtain (\ref{eq:norm7}). 
\end{proof}
\begin{prop}
\label{prop:mult-smooth}
Let $N$ be a nonnegative integer  and assume  $\zeta$ is a smooth function such
that 
\begin{equation}
\label{eq:smooth}
 |\partial^\alpha \zeta(x)|\leq C_N \quad \text{for all } |\alpha|\leq N \
\text{and } x\in \setR^n.
\end{equation}
If $u\in W_{s,\delta}^p(\setR^n)$, $|s|< N$ and $1<p<\infty$, then
\begin{equation*}
 \|\zeta u\|_{W_{s,\delta}^p(\setR^n)}\leq C_N \|u\|_{W_{s,\delta}^p(\setR^n)}.
\end{equation*}
\end{prop}
\begin{proof}
  For such smooth function 
$\|\zeta   u\|_{W_s^p}\leq C_N \|u\|_{W_s^p}$ holds.
  This inequality can be proven by interpolation, for details see the
  Lemma in \cite{triebel76:_spaces_kudrj}.
  By (\ref{eq:smooth}), $|(\partial^\alpha \zeta)(2^jx)|\leq C_N$, and
  hence
  \begin{equation*}
 \|\zeta u\|_{W_{s,\delta}^p(\setR^n)}^p=\sum_{j=0}^\infty
2^{(\delta+\frac{n}{p})pj}
\left\|\left(\psi_j \zeta u\right)_{(2^j)}\right\|_{W_s^p}^p\leq C_N^p
\sum_{j=0}^\infty 2^{(\delta+\frac{n}{p})pj}
\left\|\left(\psi_j u\right)_{(2^j)}\right\|_{W_s^p}^p
=C_N^p \| u\|_{W_{s,\delta}^p(\setR^n)}^p
\end{equation*}
\end{proof}
\begin{prop}
\label{prop:Properties:11}
Let  $ \chi_R\in C^\infty(\mathbb{R}^n)$ be a cut--off
function such that
$\chi_R(x)=1$
for $|x|\leq R$, $\chi_R(x)=0$ for $|x|\geq 2R$ and
\begin{math}
  |\partial^\alpha \chi_R|\leq c_\alpha R^{-|\alpha|}
\end{math}. Then for $\delta'<\delta$ 
\begin{equation*}
\|(1-\chi_R)u\|_{W_{s,\delta'}^p(\setR^n)}\lesssim{R^{ -(\delta-\delta') } }
\|u\|_{W_{s,\delta}^p(\setR^n)}
\end{equation*}
holds.
\end{prop}

\begin{proof}
    Let $J_0$ be the smallest  integer such that $R\leq 2^{J_0+1}$.
    Then $(1-\chi_R)\psi_j=0$ for $j=0,1,...,J_0-1$, and hence
    \begin{equation*}
      \begin{split} &\|(1-\chi_R) u\|_{W_{s,\delta'}^p(\setR^n)}^p  =
\sum_{j=J_0}^\infty 2^{\left( \delta' +\frac{n}{p} \right)pj} \left\| \left(
\psi_j(1-\chi_R) u \right)_{2^j} \right\|_{W_{s}^p}^p \\   \lesssim & 
\sum_{j=J_0}^\infty  2^{-(\delta-\delta^\prime)pj}2^{\left(\delta +\frac{n}{p}
\right)pj} \left\| \left( \psi_j u \right)_{2^j} \right\|_{W_s^p}^p \lesssim
2^{-(\delta-\delta^\prime)pJ_0} \sum_{j=J_0}^\infty 2^{\left(\delta +\frac{n}{p}
\right)pj} \left\| \left( \psi_j u \right)_{2^j} \right\|_{W_s^p}^p \\  \lesssim
& \left({R^{-(\delta-\delta')}}\|u\|_{W_{s,\delta}^p(\setR^n)}\right)^p.
\end{split}
    \end{equation*}
   
\end{proof}
The next proposition deals with embeddings.
It concerns also the embedding into the weighted space of continuously
differentiable functions, $C_{\beta}^m(\setR^n)$ where $m$ is a
nonnegative integer, $\beta\in \mathbb{R}$ and which posses the
following norm
\begin{eqnarray}  
\label{eq:norm-cont}
\displaystyle
& \|u\|_{C_\beta^m(\setR^n)}=\sum_{|\alpha|\leq
m}\sup_{\setR^n}\left((1+|x|)^{\beta+|\alpha|} |\partial^\alpha
u(x)|\right).
\end{eqnarray}
\begin{prop}[Embedding]
  \label{prop:2}$\quad$
  \begin{enumerate}
    \item[{\rm (a)}] \label{8} Let $s_1\leq s_2$ and
    $\delta_1\leq\delta_2$, then the inclusion
    $i:W_{s_2,\delta_2}^p(\setR^n)\to
W_{s_1,\delta_1}^p(\setR^n)$ is
    continuous.

    \item[{\rm (b)}]
    \label{eq:embd-compact}
    Let $s_1<s_2$ and $\delta_1<\delta_2 $, then the embedding
    \begin{math} i: W_{s_2,\delta_2}^p(\setR^n)\to
      W_{s_1,\delta_1}^p(\setR^n)
    \end{math}
    is compact.

    \item[{\rm (c)}]
    \label{eq:embd-cont}
    Let $s>\frac{n}{p} + m$ and $\delta+\frac n p\geq\beta$, then the
    embedding $i:W_{s,\delta}^p(\setR^n)\to C_\beta^m(\setR^n)$ is
    continuous.

  \end{enumerate}
\end{prop}

\begin{proof}
From the definitions of the norms (\ref{eq:norm2}) and (\ref{eq:norm1}), we see
that they are increasing functions of both $s$ and $\delta$. Hence
$\|u\|_{W_{s_1,\delta_1}^p(\setR^n)}\leq \|u\|_{W_{s_2,\delta_2}^p(\setR^n)}$
and that proves (a).
 To prove  (b), we let $N$ be a positive integer and set
$i_N(u)=\sum_{j=0}^N (\psi_j u)_{2^j}$. 
Since $i_N(u)$ has support in $\{|x|\leq 2^{N+2}\}$ and $s_1<s_2$,
$i_N: W_{s_2}^p\to W_{s_1}^p$ is a compact operator (see e.g.
\cite[\S3.3.2]{Edmunds_Triebel}\footnote{In this Theorem the authors estimate
the entropy numbers of the embedding and show that it tends to zero. This,
however, implies the compactness.}). In addition, by Proposition
\ref{prop:Properties:11}
we have
\begin{equation*}
 \|i_N(u)-i(u)\|_{W_{s_1,\delta_1}^p(\setR^n)}\lesssim
2^{-N(\delta_2-\delta_1)}
\|u\|_{W_{s_2,\delta_2}^p(\setR^n)}.
\end{equation*} 
Thus the embedding $i$ is a norm limit of compact operators, hence it is itself
 compact (\protect  see e.g. \cite[Theorem 4.11]{Schechter}).

We turn now to (c).  Assume first that $m=0$ and $s>\frac{n}{p} $, then
\begin{math}
 \sup_{\setR^n}|u(x)|\lesssim \|u\|_{W_s^p}
\end{math} (see e.~g.~\cite[\S32]{Tartar}). Applying it term--wise to the norm
(\ref{eq:norm1}), we have
\begin{equation}
\label{eq:6}
\begin{split}
 &   \sup_{\setR^n}(1+|x|)^\beta|u(x)|
 \leq  2^\beta \sup_{j\geq 0}\left(2^{\beta
j}\sup_{\setR^n}|\psi_j(x)u(x)|\right)
\\  = &  2^\beta \sup_{j\geq0}\left(2^{\beta j}\sup_{\setR^n}|\psi_j(2^j x)u(2^j
x)|\right) \lesssim 
2^\beta  \sup_{j\geq 0}\left(2^{\beta j}\|(\psi_j
u)_{2^j}\|_{W_s^p}\right)
\\  \lesssim  & 2^\beta  \sup_{j\geq0}\left(2^{(\delta+\frac{n}{p}) j}\|(\psi_j
u)_{2^j}\|_{W_s^p}\right)
  \lesssim 2^\beta  \|u\|_{W_{s,\delta}^p(\setR^n)}.
\end{split}
\end{equation}

If $m\geq1$ and $|\alpha|\leq m$, then $\partial^\alpha u\in
W_{s-|\alpha|,\delta+|\alpha|}^p(\setR^n)$ by Proposition \ref{prop:1}.
So applying (\ref{eq:6})  to $\partial^\alpha u$ with
$\delta^\prime=\delta+|\alpha|$  and $\beta^\prime=\beta+|\alpha|$, we  obtain
$ \|\partial^\alpha u\|_{C_{\beta+|\alpha|}^0(\setR^n)}\leq C \|\partial^\alpha
u\|_{W_{s-|\alpha|,\delta+|\alpha|}^p(\setR^n)}$.   
\end{proof}

For further applications we discuss the construction of the sequence
$\{\psi_j\}$ that appears  in Definition \ref{def:spaces}. Let $h$ be a
$C^\infty(\setR)$ function such that $h(t)=-1$ for
$t\leq \frac{1}{4}$, $h(t)=0$ for $1/2\leq t \leq 1$ and $h(t)=1$ for $2\leq
t$. Let
\begin{equation}
\label{eq:norm6}
 g(t)=\left\{ 
\begin{array}{ll}
e^{\frac{-t^2}{(1-t^2)}}, & |t|<1\\
0, & |t|\geq 1
\end{array}\right..
\end{equation} 
Then the functions $\psi_j(x)=g(h(2^{-j}|x|))$ satisfy the requirements of the
dyadic resolution above,  Definition \ref{def:spaces}. Moreover, for any
positive
$\gamma$, $\psi_j^\gamma(x)=g^\gamma(h(2^{-j}|x|))$ and from (\ref{eq:norm6})
we see that there
are two constants $C_1(\gamma,\alpha)$ and $C_2(\gamma,\alpha)$ such that
\begin{equation*}
 C_1(\gamma,\alpha)|\partial^\alpha \psi_j(x)|\leq |\partial^\alpha
\psi_j^\gamma(x)|\leq
 C_2(\gamma,\alpha)|\partial^\alpha \psi_j(x)|
\end{equation*}
for any multi--index $\alpha$, and these inequalities are independent of $j$.
Therefore the family
$\{\psi_j^\gamma\}$ satisfies  condition (\ref{eq:const:4})
and hence by Theorem \ref{thm1} (a) we obtain:
\begin{prop}
\label{prop:eqiv}
Let $\gamma$ be positive number, then 
 \begin{equation}
\label{eq:norm8}
 \|u\|_{W_{s,\delta}^p(\setR^n)}^p\simeq
\sum_{j=0}^\infty 2^{(\delta+\frac{n}{p})pj}
\left\|\left(\psi_j^\gamma u\right)_{(2^j)}\right\|_{W_s^p}^p.
\end{equation} 
\end{prop}

Using  Proposition \ref{prop:eqiv} we establish multiplication 
and  fractional power properties of the weighted Besov spaces.
\begin{prop}
\label{prop:4}
Assume $s\leq \min\{s_1,s_2\}$, $s_1+s_2>s+\frac{n}{p}$,  $s_1+s_2\geq
n\cdot\max\{0,(\frac{2}{p}-1)\}$ and
$\delta\leq\delta_1+\delta_2 + \frac{n}{p}$, then the multiplication 
\begin{equation*}
 W_{s_1,\delta_1}^p(\setR^n) \times  W_{s_2,\delta_2}^p(\setR^n)\to 
W_{s,\delta}^p(\setR^n)
\end{equation*}
is continuous.
\end{prop}

\begin{proof}
Let $u\in W_{s_1,\delta_1}^p(\setR^n)$ and $v\in W_{s_2,\delta_2}^p(\setR^n)$,
then by the  
 corresponding unweighed embedding results, we have
\begin{equation*}
 \|\left(\psi_j^2 u v\right)_{2^j}\|_{W_{s}^p}\lesssim  \|\left(\psi_j
u\right)_{2^j}\|_{W_{s_1}^p} \|\left(\psi_j v\right)_{2^j}\|_{W_{s_2}^p}.
\end{equation*} 
 For the proof of these types of results see
\cite[\S4.6.1]{runst96:_sobol_spaces_fract_order_nemyt}.  Set
$a_j=\|\left(\psi_j u\right)_{2^j}\|_{W_{s_1}^p}^p$ and
$b_j=\|\left(\psi_j
v\right)_{2^j}\|_{W_{s_2}^p}^p$, then by Proposition \ref{prop:eqiv} and the
Cauchy Schwarz inequality
 we have
\begin{equation*}
\begin{split}
&\| uv\|_{W_{s,\delta}^p(\setR^n)}^p \lesssim
  \sum_{j=0}^\infty 2^{\left(\delta +\frac{n}{p}\right)pj}
  \left\| \left( \psi_j^2 uv \right)_{2^j} \right\|_{W_{s}^p}^p
  \lesssim\sum_{j=0}^\infty 2^{\left(
\delta_1+\frac{n}{p}+\delta_2+\frac{n}{p}\right)pj}a_j b_j
  \\ \lesssim &
  \left(\sum_{j=0}^\infty \left(2^{\left( \delta_1+\frac{n}{p}\right)pj}
 a_j\right)^2\right)^{\frac{1}{2}}\left(\sum_{j=0}^\infty \left(2^{\left(
\delta_2+\frac{n}{p}\right)pj}
 b_j\right)^2\right)^{\frac{1}{2}}
\\ \lesssim &
  \left(\sum_{j=0}^\infty 2^{\left( \delta_1+\frac{n}{p}\right)pj}
 a_j\right)\left(\sum_{j=0}^\infty 2^{\left(\delta_2+\frac{n}{p}\right)pj}
 b_j\right)
 \lesssim
  \| u\|_{W_{s_1,\delta_1}^p(\setR^n)}^p\| v\|_{W_{s_2,\delta_2}(\setR^n)}^p.
 \end{split}
\end{equation*}
 
\end{proof}
\begin{cor}
 Let $s>\frac{n}{p}$ and $\delta\geq-\frac{n}{p}$, then the space
$W_{s,\delta}^p$ is  an algebra.
\end{cor}

Proposition \ref{prop:4}  can be  extended
to a  multiplication of  three functions and  with relaxed conditions on
the $\delta$'s.  
\begin{prop}
\label{prop:11}
Assume $s\leq \min\{s_1,s_2\}$, $s_1+s_2>s+\frac{n}{p}$,  $s_1+s_2\geq
n\cdot\max\{0,(\frac{2}{p}-1)\}$ and
$\delta\leq\delta_1+\delta_2+\delta_3 + \frac{2n}{p}$, then the multiplication 
\begin{equation*}
 W_{s_1,\delta_1}^p \times  W_{s_2,\delta_2}^p\times
W_{s_2,\delta_3}^p\to 
W_{s,\delta}^p
\end{equation*}
is continuous.
\end{prop}
\begin{proof}
Similar to the above proof, by the multiplication 
properties in the Besov
spaces, we have
\begin{equation*}
 \|\left(\psi_j^3 u vw\right)_{2^j}\|_{W_{s}^p}\lesssim  \|\left(\psi_j
u\right)_{2^j}\|_{W_{s_1}^p} \|\left(\psi_j
v\right)_{2^j}\|_{W_{s_2}^p}\|\left(\psi_j w\right)_{2^j}\|_{W_{s_2}^p}.
\end{equation*} 

Let $a_j$ and $b_j$ be as in the previous proof and let
$c_j=\|\left(\psi_j w\right)_{2^j}\|_{W_{s_2}^p}^p$.
Replacing the
Cauchy--Schwarz inequality by the  H\"older inequality, we get that
\begin{equation*}
\begin{split}
\| uvw\|_{W_{s,\delta}^p(\setR^n)}^p  &\lesssim
  \left(\sum_{j=0}^\infty \left(2^{\left( \delta_1+\frac{n}{p}\right)pj}
 a_j\right)^2\right)^{\frac{1}{2}}\left(\sum_{j=0}^\infty \left(2^{\left(
\delta_2+\frac{n}{p}\right)pj}
 b_j\right)^4\right)^{\frac{1}{4}} \left(\sum_{j=0}^\infty \left(2^{\left(
\delta_3+\frac{n}{p}\right)pj}
 c_j\right)^4\right)^{\frac{1}{4}}
\\  &\lesssim 
  \left(\sum_{j=0}^\infty 2^{\left( \delta_1+\frac{n}{p}\right)pj}
 a_j\right)\left(\sum_{j=0}^\infty 2^{\left(\delta_2+\frac{n}{p}\right)pj}
 b_j\right)\left(\sum_{j=0}^\infty 2^{\left(
\delta_3+\frac{n}{p}\right)pj}
 c_j\right)\\ &
 \lesssim
  \| u\|_{W_{s_1,\delta_1}^p(\setR^n)}^p\| 
v\|_{W_{s_2,\delta_2}^p(\setR^n)}^p\|
w\|_{W_{s_2,\delta_3}^p(\setR^n)}^p.
 \end{split}
\end{equation*}

\end{proof}
\begin{prop}
  \label{prop:5a} 
  Let $u\in W_{s,\delta}^p\cap L^\infty$, $1\leq\beta$, $
  0<s<\beta +\frac{1}{p}$ and $\delta\in \mathbb{R}$, then
  \begin{equation*}
  \||u|^\beta\|_{W_{s,\delta}^p(\setR^n)}\leq
    C(\|u\|_{L^\infty})
    \|u\|_{W_{s,\delta}^p(\setR^n)}.
  \end{equation*}
\end{prop}
\begin{proof}
The unweighted inequality, 
\begin{equation}
\label{eq:kateb} 
\||u|^\beta\|_{W_{s}^p}\leq
C(\|u\|_{L^\infty})
    \|u\|_{W_{s}^p}.
\end{equation}
was proven by Bourdaud and  Meyer \cite{Bourdaud_Meyer} for $\beta=1$ and by
Kateb \cite{kateb03:_besov} for $1<\beta$. Applying (\ref{eq:kateb}) term--wise
to the equivalent norm (\ref{eq:norm8}), we get
\begin{equation*}
    \begin{split}
     \| |u|^\beta\|_{W_{s,\delta}(\setR^n)^p}^p  &\simeq
    \sum_{j=0}^\infty 2^{( \delta+\frac{n}{p})pj}
    \| (\psi_j^\beta |u|^\beta)_{2^j} \|_{W_{s}^p}^{p}\\ & \leq
\left(
    C(\|u\|_{L^\infty})\right)^p
    \sum_{j=0}^\infty 2^{(  \delta+\frac{n}{p})pj}
    \| (\psi_j u)_{2^j}\|_{W_{s}^p}^{p}\leq \left(
    C(\|u\|_{L^\infty})\right)^p \| u \|_{W_{s,\delta}^p(\setR^n)}^p.
    \end{split}
\end{equation*}
 
\end{proof}
\begin{prop}
 \label{prop:Moser}
Let $F:\setR^m\to\setR^l$ be a $C^{N+1}$ function such that $F(0)=0$ and 
$0<s\leq N$.  Then there exists a positive constant $C$ such that  
\begin{equation}
\label{eq:Moser:2}
 \|F(u)\|_{W_{s,\delta}^p(\setR^n)}\leq C
\|F\|_{C^{N+1}}\left(1+\|u\|_{L^\infty}^N\right)\|u\|_{W_{s,\delta}^p(\setR^n)}
\end{equation} 
for any $u\in
W_{s,\delta}^p(\setR^n)\cap L^\infty(\setR^n)$.
In particular, if $s>\frac{n}{p}$ and  $\delta\geq-\frac{n}{p}$, then
\begin{equation}
\label{eq:Moser}
 \|F(u)\|_{W_{s,\delta}^p(\setR^n)}\leq C
\|u\|_{W_{s,\delta}^p(\setR^n)}.
\end{equation}
\end{prop}
\begin{proof}
The Moser type estimate 
\begin{equation}
\label{eq:norm12}
 \|F(u)\|_{W_{s}^p}\leq C
\|F\|_{C^{N+1}}\left(1+\|u\|_{L^\infty}^N\right)\|u\|_{W_{s}^p}, 
\end{equation} 
in the Besov spaces was proven in
\cite[\S5.3.4]{runst96:_sobol_spaces_fract_order_nemyt}.
Let $\{\psi_j\}$ be the dyadic resolution of unity used in the
definition of the norm (\ref{eq:norm1}) and set $\Psi_j(x)=
(\varphi(x))^{-1}\psi_j(x)$, where $\varphi(x)=\sum_{j=0}^\infty
\psi_j(x)$. Then   the sequence
$\{\Psi_j\}\subset C_0^\infty(\setR^n)$, satisfies (\ref{eq:const:4}) and
$\sum_{j=0}^\infty\Psi_j(x)=1$.  Since $F(0)=0$, we obtain
\begin{equation}
\label{eq:norm11}
 \left(\psi_jF(u)\right)_{2^j}=
 \left(\psi_jF\left(\sum_{k=0}^{\infty}\Psi_ku\right)\right)_{2^j }=
 \left(\psi_jF\left(\sum_{k=j-2}^{j+1}\Psi_ku\right)\right)_{2^j},
\end{equation}
for each $j$. Here we use the convention that a summation starts
from zero whenever $k<0$.
By the Moser type estimate (\ref{eq:norm12}), we have
\begin{equation}
\label{eq:norm15}
 \left\|\left(\psi_jF\left(\sum_{k=j-2}^{j+1}\Psi_ku\right)\right)_{2^j 
}\right\|_{W_s^p}\leq C 
\left\|\left(\sum_{k=j-2}^{j+1}\Psi_ku\right)_{2^j}\right\|_{W_s^p}
\leq C 
\sum_{k=j-2}^{j+1}\left\|\left(\Psi_ku\right)_{2^j}\right\|_{W_s^p},
\end{equation} 
where $C=C(\|F\|_{C^{N+1}}, \|u\|_{L^\infty})$. Taking into account the known 
 scaling properties of the Besov's norm and the form of $\Psi_k$, we get that 
\begin{equation}
  \left\|\left(\Psi_ku\right)_{2^j}\right\|_{W_s^p}=
  \left\|\left(\left(\Psi_ku\right)_{2^k}\right)_{2^{j-k}}\right\|_{W_s^p}
  \lesssim 
2^{(k-j)n/p}2^{2s}\left\|\left(\Psi_ku\right)_{2^k}
\right\|_{W_s^p}.
\end{equation}
Combining  (\ref{eq:norm11}), (\ref{eq:norm15}) with inequality
 $\left\|\left(\Psi_ku\right)_{
2^k }\right\|_{W_s^p}\leq C\left\|\left(\psi_ku\right)_{
2^k }\right\|_{W_s^p}$, we obtain that
\begin{equation*}
 \begin{split}
 &\ \ \  \left\|F(u)\right\|_{W_{s,\delta}^p(\setR^n)}^p =\sum_{j=0}^\infty
2^{(\delta+\frac{n}{p})pj}\left\|\left(\psi_j
F(u)\right)_{2^j}\right\|_{W_s^p}^p
\\ & \leq
\left(C(\|F\|_{C^{N+1}},\|u\|_{L^\infty}^N)\right)^p
\sum_{j=0}^\infty 2^{(\delta+\frac{n}{p})pj}\sum_{k=j-2}^{j+1}
2^{(k-j)n}\left\|\left(\psi_k
u\right)_{2^k}\right\|_{W_{s}^p}^p
\\ &\leq 4
\left(C(\|F\|_{C^{N+1}},\|u\|_{L^\infty}^N)\right)^p
\sum_{k=0}^{\infty} 2^{(\delta+\frac{n}{p})pk}
\left\|\left(\psi_ku\right)_{2^k}\right\|_{W_{s}^p}^p
\\ & = 4
\left(C(\|F\|_{C^{N+1}},\|u\|_{L^\infty}^N)\right)^p
\left\|u\right\|_{W_{s,\delta}^p(\setR^n)}^p.
 \end{split}
\end{equation*}
When $s>\frac{n}{p}$ and  $\delta\geq\frac{n}{p}$, then (\ref{eq:Moser}) follows
from Proposition
\ref{prop:2}(c).
 
\end{proof}

\section{Linear Elliptic Systems on Asymptotically Flat Riemannian Manifolds}
\label{sec:linear_asymp}
In this section we study second order linear elliptic systems whose
coefficients are in the weighted Besov spaces. We emphasize the study
of operators with the Laplace Beltrami operator of an asymptotically flat
Riemannian manifold as the principal part.  The range of $\delta$ is
restricted to the interval $(-\frac{n}{p},-2+\frac{n}{p'})$, since for
these values of $\delta$ the Laplace operator is an isomorphism
between $W_{s,\delta}^p(\setR^n)$ and $W_{s-2,\delta+2}^p(\setR^n)$.

\subsection{Linear elliptic operators in $\setR^n$}
\label{sec:ellipt-oper-w_s}
We consider second order  linear elliptic systems  of the form
\begin{equation}
\label{eq:ellptic:1}
(L u)^i = (a_2)_{ij}^{ab}\partial_a\partial_bu^j+(a_1)_{ij}^a\partial_a u^j 
+(a_0)_{ij}u^j,
\end{equation} 
where $a_k$ are $N\times N$ block matrices. 
The operator $L$ is elliptic  in $\setR^n$ if 
\begin{equation}
\label{eq:elliptic11}
 \det\left((a_2)_{ij}^{ab}(x)\xi_a\xi_b\right)\neq0\qquad \text{ for all}\
 \xi\in \setR^n\setminus\{0\}  \ \text{and}\ x\in\setR^n.
\end{equation} 
Let $A_\infty$ be a matrix with constant coefficients, the symbol $A_\infty $
stands also
for a second order differential operator  of the form
\begin{equation}
\label{eq:elliptic:2}
(A_\infty 
u)^i=(A_\infty)_{ij}^{ab}\partial_a\partial_b 
u^j.
\end{equation} 
We  assume
$\det\left(\left(A_\infty\right)_{ij}^{ab}\xi_a\xi_b\right)\neq0$, hence
$A_\infty$ is an elliptic operator.
\begin{defn}
 \label{def:asy}
We say that  operator $L$ belongs ${\bf 
Asy}(A_\infty,s,\delta,p)$ if
\begin{equation}
\label{eq:elliptic3}
 a_2-A_\infty\in W_{s,\delta}^p(\setR^n), \quad a_1\in
W_{s-1,\delta+1}^p(\setR^n)
\quad\text{and} \quad a_0\in
W_{s-2,\delta+2}^p(\setR^n).
\end{equation}  
\end{defn}

The following Corollary is a consequence of Propositions \ref{prop:1} and
\ref{prop:4}. 
\begin{cor}
 \label{cor:elliptic1}
Let $L\in({\bf 
Asy}A_\infty,s,\delta,p)$, $s\in(\frac{n}{p},\infty)\cap
[1,\infty)$, $\delta\in 
[-\frac{n}{p},\infty)$ and
$p\in(1,\infty)$, then 
\begin{equation*}
 L:W_{s,\delta}^p(\setR^n)\to W_{s-2,\delta+2}^p(\setR^n)
\end{equation*}  
is a bounded operator.
\end{cor}
\begin{lem}
 \label{lem:elliptic1}
 Let $\delta\in (-\frac{n}{p}, -2+\frac{n}{p'})$ and $p\in(1,\infty)$,
 $s\in\setR$, 
 then the operator
 \begin{equation}
\label{eq:ellliptic2}
A_\infty :W_{s,\delta}^p(\setR^n)\to W_{s-2,\delta+2}^p(\setR^n) 
\end{equation} 
is an isomorphism.
\end{lem}
\begin{proof}(of Lemma \ref{lem:elliptic1})
{For an integer $s$ that is
greater or equal two the   isomorphism of system
(\ref{eq:ellliptic2}) 
was proven by Lockhart  and McOwen} 
 \cite[Theorem 3]{Lockhart_McOwen}\footnote{ Lockhart and
    McOwen considered the Douglis Nirenberg elliptic system, which is
    an extension of the common elliptic systems (\ref{eq:ellptic:1}).
The ellipticity of this system is expressed by means of two $N$ tuples
$(s_1,\ldots,s_N)$ and $(t_1,\ldots,t_N)$ of nonnegative integers.
Setting $s_i=m$ and $t_i=m+2$ for $i=1,\ldots,N$, reduces the constant
coefficients Douglis Nirenberg system of \cite{Lockhart_McOwen} to
(\ref{eq:elliptic:2}).}
Hence, by interpolation, Theorem \ref{thm:Triebel}(d), it is an
isomorphism for all $s\geq 2$.

For $s\leq 2$ we shall consider the adjoint operator. 
Let $u$ and $\varphi$ be two smooth functions, then  
$((A_\infty)_{ij}^{ab}\partial_a\partial_b
u,\varphi)_{L^2(\setR^n)}=(u,(A_\infty)_{ji}^{ba}\partial_a\partial_b\varphi)_{
L^2(\setR^n)} $, and by Proposition \ref{prop:dual2} and approximation, 
\begin{math}
\langle (A_\infty)_{ij}^{ab}\partial_a\partial_b
u,\varphi\rangle_W=\langle u,(A_\infty)_{ji}^{ba}
\partial_a\partial_b\varphi\rangle_W
\end{math}. Thus letting $A^\ast_\infty$ denote the adjoint matrix of
$A_\infty$, we
conclude from Theorem \ref{thm:Triebel} (c) that 
 \begin{equation}
\label{eq:ellliptic3}
A_\infty^\ast:W_{-s+2,-\delta-2}^{p^\prime}(\setR^n)\to
W_{-s,-\delta}^{p^\prime}(\setR^n) 
\end{equation}  
is the adjoint operator of (\ref{eq:ellliptic2}). 
Note  that 
$-\frac{n}{p^\prime}<-\delta-2<-2+\frac{n}{p}$,  so the previous part implies
that
(\ref{eq:ellliptic3}) is an isomorphism for $s\leq 0$. 
Since the adjoint of an isomorphism   is also an
isomorphism (see e.g. \cite[Theorem 5.15]{Schechter}), we conclude that
(\ref{eq:ellliptic2}) is an isomorphism for all negative integers, and by
interpolation for all $s$.
\end{proof}
    
In order to prove \textit{a priori} estimates for 
$L\in ({\bf A}_\infty,s,\delta,p)$ we need the corresponding 
result in the unweighted 
Besov spaces. The following Lemma was proven in
\cite[Lemma 32]{Holst_Nagy_Tsogtgerel}.
\begin{lem}[Holst, Nagy and Tsogtgerel]
\label{lem:3}
Assume that the coefficients of $L$
satisfy the conditions: $a_i\in W_{s-i}^p$
for $i=0,1,2$, $s\in(\frac{n}{p},\infty)\cap [1,\infty)$ and
$p\in(1,\infty)$, and (\ref{eq:elliptic11}).
Then for all $u\in W_s^p$ with support in a compact set $K$, there
exists 
a constant $C$ that depends on $K$ and the
$W_{s-i}^p$-norms of the coefficients $a_i$ such that
\begin{equation}
 \left\|u\right\|_{W_s^p}\leq C\left\{ \left\|Lu\right\|_{W_{s-2}^p}+
\left\|u\right\|_{W_{s-1}^p}\right\}.
\end{equation}
\end{lem}

\begin{lem}
 \label{lem:elliptic2}
 Let $L\in{\bf Asy}(A_\infty,s,\delta,p)$,
 $s\in(\frac{n}{p},\infty)\cap [1,\infty)$, $\delta \in
 (-\frac{n}{p},-2+\frac{n}{p^\prime})$, $p\in(1,\infty)$ and
 $\delta'<\delta$.
 Then for any $u\in W_{s,\delta}^p(\setR^n)$,
 \begin{equation}
\label{eq:elliptic12}
 \left\|u\right\|_{W_{s,\delta}^p(\setR^n)}\leq C\left\{
\left\|Lu\right\|_{W_{s-2,\delta+2}^p(\setR^n)}+
\left\|u\right\|_{W_{s-1,\delta'}^p(\setR^n)}\right\},
\end{equation}
where the constant $C$ depends on  $W_{s,\delta}^p$-norms of the coefficients of
$L$, $s,\delta,p$ and $\delta'$.
\end{lem}

A continuous linear operator $L:E\to F$, where $E$ and $F$ are Banach spaces, 
is called Semi--Fredholm if its kernel is finite dimensional and it has a closed
range. This is equivalent to the inequality 
\begin{equation*}
 \|u\|_{E}\leq C \left\{ \|Lu\|_{F}+| u |\right\},
\end{equation*}
where the norm $ |\cdot  |$ is  compact relative to the norm $\|  \cdot
\|_E$ (see e.g.~\cite[Theorem 6.2]{Schechter}).

Thus as a consequence of the estimates (\ref{eq:elliptic12}) and the compact
embedding Proposition \ref{prop:2}(b), we obtain:
\begin{cor}[Semi--Fredholm]
\label{cor:1}
 Let us assume that the conditions of Lemma \ref{lem:elliptic2} hold, then
$L:W_{s,\delta}^p(\setR^n)\to W_{s-2,\delta+2}^p(\setR^n)$ is a semi--Fredholm
operator. 
\end{cor}

\textit{Proof of Lemma \ref{lem:elliptic2}.}
Let $\chi_\rho$ be a cut--off function such that $\supp{(\chi_\rho)}\subset
B_{2\rho}$, $\chi_\rho(x)=1$  on $B_\rho$ and $|\partial^\alpha
\chi_\rho|\leq C_\alpha \rho^{-|\alpha|}$. Here $B_\rho$ denotes a ball of
radius $\rho$. We decompose  $u=(1-\chi_\rho)u+\chi_\rho u$ and estimate each
term separately.  Lemma
\ref{lem:elliptic1} implies that $A_\infty$ is an isomorphism, hence there is a
constant $C$ such that
\begin{equation}
\label{eq:elliptic14}
 \left\|(1-\chi_\rho)u\right\|_{W_{s,\delta}^p(\setR^n)}\leq C
\left\|A_\infty\left((1-\chi_\rho)u\right)\right\|_{W_{s-2,\delta+2}^p(\setR^n)}
.
\end{equation}
Let  $[A_\infty,(1-\chi_\rho)]$ denote a commutation, that is,
\begin{equation}
\label{eq:elliptic:15}
\left([A_\infty,(1-\chi_\rho)]u\right)^i=-\left(A_\infty\right)^{ab}_{ij}
\left\{\partial_a \partial_b\chi_\rho u^j+2\partial_a\chi_\rho\partial_b 
u^j\right\}.
\end{equation} 
Then
\begin{equation}
\label{eq:elliptic17}
 A_\infty\left((1-\chi_\rho)u\right)=[A_\infty,(1-\chi_\rho)]u
+(1-\chi_\rho)L(u)
-(1-\chi_\rho)(L-A_\infty)u.
\end{equation}
The coefficients of the commutator (\ref{eq:elliptic:15}) have  compact 
support, hence we may replace $\delta+2$ in its 
$W_{s-2,\delta+2}^p(\setR^n)$--norm by any other $\delta'$, and that will 
result in an equivalent norm.  So by Propositions
\ref{prop:1} and  
\ref{prop:mult-smooth} we obtain that
\begin{equation}
\label{eq:elliptic16}
\begin{split}
 \left\|[A_\infty,(1-\chi_\rho)]u\right\|_{W_{s-2,\delta+2}^p(\setR^n)} & \leq
C_1(\rho)\left\{\left\|u\right\|_{W_{s-2,\delta'}^p(\setR^n)}+
\left\|Du\right\|_{W_{s-2,\delta'+1}^p(\setR^n)}\right\} \\ &
\leq C_1(\rho)\left\|u\right\|_{W_{s-1,\delta'}^p(\setR^n)}.
\end{split}
\end{equation} 
Letting $\delta_1=-\frac{n}{p}$ and $\delta_2=\delta+2$ allow us to apply
Proposition \ref{prop:4} and with a combination of Proposition \ref{prop:1},
we get  that
\begin{equation*}
\begin{split}
\left\|(1-\chi_\rho)(A_\infty
-a_2)D^2u\right\|_{W_{s-2,\delta+2}^p(\setR^n)} &\lesssim
\left\|(1-\chi_\rho)(A_\infty -a_2)\right\|_{W_{s,\delta_1}^p(\setR^n)}
\left\|D^2u\right\|_{W_{s-2,\delta+2}^p(\setR^n)}\\
 & \lesssim
\left\|(1-\chi_\rho)(A_\infty -a_2)\right\|_{W_{s,\delta_1}^p(\setR^n)}
\left\|u\right\|_{W_{s,\delta}^p(\setR^n)}.
\end{split}
\end{equation*}
 Since $\delta>\delta_1=-\frac{n}{p}$, we can apply  Proposition
\ref{prop:Properties:11} and obtain  that 
\begin{equation*}
 \left\|(1-\chi_\rho)(A_\infty -a_2)\right\|_{W_{s,\delta_1}^p(\setR^n)}
\lesssim \rho^{-(\delta-\frac{n}{p})}\left\|(A_\infty
-a_2)\right\|_{W_{s,\delta}^p(\setR^n)}.
\end{equation*}
Repeating similar arguments with the other terms, we conclude that
\begin{equation}
\label{eq:elliptic22}
\left\|(1-\chi_\rho)(L-A_\infty)u\right\|_{W_{s-2,\delta+2}^p(\setR^n)}
\leq\rho^{-(\delta-\frac{n}{p})} \Lambda
\left\|u\right\|_{W_{s,\delta}^p(\setR^n)},
\end{equation} 
where
\begin{equation*}
 \Lambda\simeq\left\|(a_2-A_\infty)\right\|_{W_{s,\delta}^p(\setR^n)}
+\left\|a_1\right\|_{W_{s-1,\delta+1}^p(\setR^n)}
+ \left\|a_0\right\|_{W_{s-2,\delta+2}^p(\setR^n)}.
\end{equation*}
Thus from inequalities (\ref{eq:elliptic14}),  (\ref{eq:elliptic16}) and
(\ref{eq:elliptic22}),  and
the identity (\ref{eq:elliptic17}), we obtain that
\begin{equation}
 \label{eq:elliptic20}
\begin{split}
 \left\|(1-\chi_\rho)u\right\|_{W_{s,\delta}^p(\setR^n)} &\leq C
 \left\|Lu\right\|_{W_{s-2,\delta+2}^p(\setR^n)}+C_1(\rho)\left\|u\right\|_{W_{
s-1,\delta'}^p(\setR^n)}\\ & +\rho^{-(\delta-\frac{n}{p})} \Lambda
\left\|u\right\|_{W_{s,\delta}^p(\setR^n)}.
\end{split}
\end{equation}  

We turn now to the second term. Since $\chi_\rho u$ has compact support,
$\|\chi_\rho u\|_{W_{s,\delta}^p(\setR^n)}\simeq \|\chi_\rho u\|_{W_{s}^p}$, so by 
Lemma \ref{lem:3},
\begin{equation}
\label{eq:elliptic13}
\|\chi_\rho u\|_{W_{s,\delta}^p(\setR^n)}\simeq \|\chi_\rho u\|_{W_{s}^p}
\leq C\left\{\|L(\chi_\rho u)\|_{W_{s-2}^p}+\|\chi_\rho u\|_{W_{s-1}^p}\right\}.
\end{equation} 
Now  $L(\chi_\rho u)=\chi_\rho Lu+[L,\chi_\rho]u$, where the commutator 
$[L,\chi_\rho]$ is an operator of
order one and with coefficients with compact support in $B_{2\rho}$. 
Thus as in the  estimate of the commutator (\ref{eq:elliptic:15}),  similar 
arguments provide that
\begin{equation}
\label{eq:elliptic21}
\begin{split}
\|L(\chi_\rho u)\|_{W_{s-2}^p} & \leq \|\chi_\rho
(Lu)\|_{W_{s-2}^p}+\|[L,\chi_\rho]u\|_{W_{s-2}^p}\\ &\leq
C_2(\rho)\left\{\|Lu\|_{W_{s-2,\delta+2}^p(\setR^n)}+\|u\|_{W_{s-1,\delta'}
^p(\setR^n)} \right\}.
\end{split}
\end{equation}
Combining inequalities (\ref{eq:elliptic20}),  (\ref{eq:elliptic13}) and
(\ref{eq:elliptic21})   yields
\begin{equation*}
\begin{split}
 \| u\|_{W_{s,\delta}^p(\setR^n)} &\leq
\left\{C+C_2(\rho)\right\}\| L u\|_{W_{s-2,\delta+2}^p(\setR^n)}
+\left\{C_1(\rho)+C_2(\rho)\right\}\| u\|_{W_{s-1,\delta'}^p(\setR^n)}
\\ &+\rho^{-(\delta-\frac{n}{p})}\Lambda \| u\|_{W_{s,\delta}^p(\setR^n)}.
\end{split}
\end{equation*} 
Thus choosing $\rho$ sufficiently large such that
$\rho^{-(\delta-\frac{n}{p})}\Lambda\leq\frac{1}{2}$ completes the proof.
\hfill{$\square$}

The next proposition asserts that solutions to the homogeneous equation have
 a lower growth at infinity. 
\begin{prop}
\label{prop:10}
Assume $L\in{\bf Asy}(A_\infty,s,\delta,p)$ $s\in
(\frac{n}{p},\infty)\cap [1,\infty)$ and
$\delta\in(-\frac{n}{p},-2+\frac{n}{p'})$.
If
\begin{math}
 L u=0,
\end{math}
then $u\in W_{s,\delta'}^p(\setR^n)$ for any $\delta'\in
(-\frac{n}{p},-2+\frac{n}{p'})$.
\end{prop}
\begin{proof}
We follow the idea of Christodoulou and O'Murchadha \cite{OMC}.
By Proposition \ref{prop:2} (a) it suffices to prove  the statement  for
$\delta'>\delta$. Let
\begin{equation*}
 f=(L-A_{\infty})u.
\end{equation*}
At the first stage we chose $\delta'>\delta$ so that
$\frac{n}{p}+\delta+(\delta+2)\geq \delta'+2$.
Then by  Proposition \ref{prop:4} we obtain that
\begin{equation*}
 \left\| f\right\|_{W_{s-2,\delta'+2}^p(\setR^n)}\lesssim\left(\left\|
a_2-A_\infty\right\|_{W_{s,\delta}^p(\setR^n)}+\left\|
a_1\right\|_{W_{s-1,\delta+2}^p(\setR^n)}+\left\|
a_0\right\|_{W_{s-2,\delta+2}^p(\setR^n)}\right) \left\|
u\right\|_{W_{s,\delta}^p(\M)}. 
\end{equation*}
Since $Lu=0$, $A_\infty u=f$, so by Lemma
\ref{lem:elliptic1} we obtain  that 
$\left\|u\right\|_{W_{s,\delta'}^p(\setR^n)}\lesssim 
\left\|f\right\|_{W_{s-2,\delta'+2}^p(\setR^n)}$. We now may repeat this
procedure with  $\delta'$ replacing 
 $\delta$ and  $\delta''$ replacing  $\delta'$, which  can be done iteratively
until $\delta''=-2+\frac{n}{p}$.
\end{proof}

\subsection{Asymptotically flat manifold }
\label{sec:Asymptotically flat manifold}

A Riemannian manifold $(\M,g)$ is asymptotically flat  if the 
Riemannian metric tends to the Euclidean metric at infinity. In case the 
Riemannian metric $g$ is smooth, then  is often required that the convergence to 
the Euclidean metric has a certain decay rate at infinity, as for example in 
\cite{Lee_Parker_87}. Our definition follows essentially  the one of Bartnik 
\cite{bartnik86}.  
\begin{defn}
\label{def:asymp}
Let $\M$ be $n$ dimensional smooth connected manifold and let $g$ be a
metric on $\M$ such that $(\M,g)$ is complete. We say that $(\M,g)$ is
\textbf{asymptotically flat} of the class $W_{s,\delta}^p$, if $g\in
W_{s,{\rm loc}}^p(\M)$ and there is a compact set $\mathcal{K}\subset \M$ such
that:
\begin{enumerate}
\item There is a finite collection of charts $\{(U_i,\phi_i)\}_{i=1}^N$ which
covers $\M\setminus \mathcal{K}$;
\item For each $i$,  $\phi_i^{-1}(U_i)=E_{r_i}:=\{x\in \setR^n:
|x|>r_i\}$ for some positive $r_i$;
 \item  
\label{item:3}
The pull--back $(\phi_i^\ast g )$ is uniformly equivalent to
the Euclidean  metric $e$ on $E_{r_i}$;
 \item 
\label{item:4}
For each $i$, $(\phi_i^\ast g )_{ab}-\delta_{ab}\in
W_{s,\delta}^p(E_{r_i})$, where $\delta_{ab}$ the is the Kronecker symbol.
\end{enumerate}
\end{defn}

The weighted Sobolev space on $\M$ is denoted by $W_{s,\delta}^p(\M)$
and its norm is defined as follows.
Let $(V_j,\varTheta_j)$ be
collections of charts which cover $\mathcal{K}$ and where
$\varTheta_j$ is a diffeomorphism between a ball $B_j$
in $\setR^n$ and $V_j\subset\mathcal{K}$. 
Let $\{\chi_i,\alpha_j\}$ be a partition of unity subordinate to
$\{U_i,V_j\}$, then
\begin{equation}
\label{eq:norm-M}
 \left\|u \right\|_{W_{s,\delta}^p(\M)}:= 
\sum_{i=1}^N\left\|\phi_i^\ast(\chi_i u)
\right\|_{W_{s,\delta}^p(\setR^n)}+\sum_{j=1}^{N_0}
\left\|\varTheta_j^\ast(\alpha_j u)
\right\|_{W_{s}^p(\setR^n)}
\end{equation} 
is a norm of the weighted Besov space $W_{s,\delta}^p(\M)$.  Note that the 
norm (\ref{eq:norm-M}) depends on the partition of unity, but different
partitions result in equivalent norms. A bilinear form on
$W_{s,\delta}^p(\M)\otimes W_{-s,-\delta}^{p'}(\M)$ is
defined in a similar way:
\begin{equation}
\label{eq:bilinear:1}
 \langle u, \varphi\rangle_{\M}=\sum_{i=1}^N\langle\phi_i^\ast(\chi_i u),
\phi_i^\ast(\chi_i\varphi)\rangle_W
+\sum_{j=1}^{N_0}\langle\varTheta_j^\ast(\alpha_j u),\varTheta_j^\ast(\alpha_j
\varphi)\rangle,
\end{equation} 
 where $\langle \cdot,\cdot\rangle$ is a bilinear form (\ref{eq:3}) and 
$\langle \cdot,\cdot\rangle_W$ is the bilinear form (\ref{eq:4}).
 Using (\ref{eq:3}) and inequality
(\ref{eq:norm:8}), and 
combining these  with an elementary inequality and  the norm \ref{eq:norm-M}, we
see that 
\begin{equation}
\label{eq:bilinear:2}
 |\langle u, \varphi\rangle_{\M}|\leq C \left\|u
\right\|_{W_{s,\delta}^p(\M)}\left\|\varphi \right\|_{W_{-s,-\delta}^{p'}(\M)}.
\end{equation}  
The form (\ref{eq:bilinear:1}) depends on the partition of unity, but each one
induces a topological dual $\left(W_{s,\delta}^p(\M)\right)^\ast$ isomorphic to
$ W_{-s,-\delta}^{p'}(\M)$ (see Appendix \S\ref{Appendix}).

The norm of the spaces $C_{\beta}^m(\M)$  is defined in an analogous manner,
(\ref{eq:norm-M}), that is,
\begin{equation}
\label{eq:norm-M-C}
 \left\|u \right\|_{C_{\beta,m}(\M)}:= 
\sum_{i=1}^N\left\|\phi_i^\ast(\chi_i u)
\right\|_{C_{\beta}^m(\setR^n)}+\sum_{j=1}^{N_0}
\left\|\varTheta_j^\ast(\alpha_j u)
\right\|_{C_{m}(\setR^n)}.
\end{equation} 
It follows from Condition 3~of Definition \ref{def:asymp} that  both  types of 
the norms do not depend on the metric. This fact is a known  property of the 
ordinary Sobolev spaces on compact Riemannian manifolds (see
\cite[Proposition 2.3]{hebey_1996}).

Let $E$ and $F$ be two smooth vector bundles over a Riemannian manifold 
$(\M,g)$.  A second order linear 
differential operator $L$ on $\M$ is a linear map from $C^\infty(E)$ to 
$C^\infty(F)$ that can be written in local coordinates in the form 
\begin{equation}
\label{eq:ellptic:11}
(L u)^i = (a_2)_{ij}^{ab}\nabla_a\nabla_bu^j+(a_1)_{ij}^a\nabla_a u^j 
+(a_0)_{ij}u^j,
\end{equation} 
where $\nabla$ is the covariant derivative and the coefficients $a_k$
 are tensors. 
 The operator is elliptic at $x\in \M$, if
$\det\left((a_2(x))_{ij}^{ab}\xi_a\xi_b\right)\neq0$
for all covectors $\xi\neq0$.  

The operator $L$ belongs to 
${\bf Asy}(A_\infty,s,\delta,p)$ on  $(\M,g)$ asymptotically flat manifold  if 
 $(a_2-A_\infty)\in W_{s,\delta}^p(\M)$, $a_1\in 
W_{s-1,\delta+1}^p(\M)$ and $a_0\in 
W_{s-2,\delta+2}^p(M)$,  where $A_\infty$ is a constant tensor.

The properties of the $W_{s,\delta}^p(\setR^n)$ spaces proven in Sections
\ref{sec:some-properties-h_s} and \ref{sec:ellipt-oper-w_s} are also valid for
$W_{s,\delta}^p(\mathcal{M})$. These can be proven
by using a finite  covering of the manifold  and a partition
of unity
subordinate to the covering. We will discuss some of these in the Appendix (\S
\ref{Appendix}).

Let $L=\Delta_g$ be the Laplace Beltrami operator on
the Riemannian manifold  
$\M$, then in   local coordinates it  has  the form
\begin{equation}
\label{eq:elliptic:14}
 \Delta_g u=g^{ab}\nabla_a\nabla_b u=g^{ab}\partial_a\partial_b 
u+\partial_b(g^{ab})\partial_a
u+\dfrac{1}{2}\left(g^{ab}(\partial_b g_{ab})\right)g^{ab}\partial_a u,
\end{equation} 
where $g^{ab}$ denote the inverse matrix of $g_{ab}$.  
The principle symbol is $g^{ab}\xi_a\xi_b=|\xi|_g^2$, so  it is obviously an 
elliptic operator on the manifold $\M$. For each of the charts $(U_i,\phi_i)$ 
we get from Propositions \ref{prop:1}, \ref{prop:4} and \ref{prop:Moser}, and 
condition 4.~of Definition 
\ref{def:asymp}, that 
$(g^{ab}-\delta^{ab})\in W_{s,\delta}^p(E_{r_i})$. Also  the first order terms 
belong to $W_{s-1,\delta+1}^p(E_{r_i})$. Hence it follows from Definition 
\ref{def:asymp} and the norm (\ref{eq:norm-M}) that for any $a_0\in 
W_{s-2,\delta+2}^p(\M)$,
the operator
\begin{equation*}
 L=-\Delta_g + a_0
\end{equation*}
belongs  to ${\bf{Asy}}(-\Delta,s,\delta,p)$
on $(\M,g)$. Here
$\Delta u=\partial^a\partial_a u$ in local coordinates.

The coefficients of the Laplacian (\ref{eq:elliptic:14}) in the 
exterior of the ball $E_{r_i}$ can be extended to the 
entire space $\setR^n$ so that the extension  remains an elliptic operator. 
Hence, if  $s\in(\frac{n}{p},\infty)\cap [1,\infty)$ and
$\delta\geq -\frac{n}{p}$, then
  as a consequence of the above definition of the norm (\ref{eq:norm-M})
and Lemma \ref{lem:elliptic2} on manifolds (see Proposition \ref{lemma:3.5} in
Appendix \S\ref{Appendix})) we 
obtain:
\begin{cor}
\label{cor:2}
 Let $s\in(\frac{n}{p},\infty)\cap [1,\infty)$, $\delta\in
(-\frac{n}{p},-2+\frac{n}{p^\prime})$, $a_0\in W_{s-2,\delta+2}^p(\M)$ 
and assume $(\M,g)$ is an asymptotically flat manifold  of the class
$W_{s,\delta}^p$. Then
\begin{equation*}
 -\Delta_g+a_0:W_{s,\delta}^p(\M)\to W_{s-2,\delta+2}^p(\M)
\end{equation*} 
is a semi--Fredholm operator.
\end{cor}

\subsection{Weak solutions of  linear systems on manifolds}
\label{sec:weak}

In this subsection we consider weak solution of
the equation $-\Delta_g u+a_0u=f$, where $\Delta_g$ is the
Laplace Beltrami operator (\ref{eq:elliptic:14}) and $a_0,f\in
W_{s-2,\delta+2}^p(\M)$.

But prior to the definition of weak solutions, we have to extend the
$L^2$-form
\begin{equation}
\label{eq:l-2}
(u,v)_{(L^2,g)}:= \int_{\M} uv d\mu_g, \quad u,v\in C_0^\infty(\M),
\end{equation}
 to a continuous  bilinear form on 
\begin{math}
W_{s,\delta}^p(\M)\otimes W_{-s,-\delta}^{p'}(\M)
\end{math}.  Here $\mu_g$ is the volume element with respect to the metric $g$.

So assume $(\M,g)$ is an asymptotically flat manifold of the class
$W_{s,\delta}^p$, $s>\frac{n}{p}$ and $\delta\geq -\frac{n}{p}$, let
$|g|$ denote the determinant of $g$ and
$\{\chi_i,\alpha_j\}$ be the partition of unity as in the definition
of the norm (\ref{eq:norm-M}), then
\begin{math}
\{ \tilde{\chi}_i^2,\tilde{\alpha}_j^2\} :=
\left(\sum_{i=1}^N\chi_i^2+\sum_{j=1}^{N_0}\alpha_j^2\right)^{-1}\{\chi_i^2,
\alpha_j^2\}
\end{math} 
is also a partition of unity.
Using Propositions \ref{prop:dual} and \ref{prop:dual2} and the
bilinear form (\ref{eq:bilinear:1}), we obtain that
\begin{equation}
\label{eq:bilinear:4}
\begin{split}
(u,v)_{(L^2,g)} &=\sum_{i=1}^{N}
\int_{E_{r_i}}\phi^\ast_i\left((\tilde{\chi}_i u
\sqrt{|g|})(\tilde{\chi_i} v)\right)dx+\sum_{j=1}^{N_0}
\int_{B_j}\varTheta^\ast_j\left((\tilde{\alpha_j} u \sqrt{|g|})(\tilde{\alpha_j}
v)\right)dx\\ &=
\sum_{i=1}^{N}
\langle\phi^\ast_i(\tilde{\chi}_i u
\sqrt{|g|}), \phi^\ast_i(\tilde{\chi_i} v)\rangle_W+\sum_{j=1}^{N_0}
\langle \varTheta^\ast_j(\tilde{\alpha_j} u \sqrt{|g|}),
\varTheta^\ast_j(\tilde{\alpha_j}
v)\rangle\\ &=\langle
\sqrt{|g|}u,v\rangle_{\M}=:\langle u,v\rangle_{(\M,g)}.
\end{split}
\end{equation} 
Now 
it follows from the generalised H\"older inequality (\ref{eq:bilinear:2}) and
Propositions \ref{prop:4} (and its corresponding form on manifolds, Proposition 
\ref{prop:4M}) and \ref{prop:Moser} that whenever  $s-2\leq s'\leq s$ and 
any $\delta'$, then
\begin{equation}
\begin{split}
 |\langle u,v\rangle_{(\M,g)}| & =|\langle
\sqrt{|g|}u,v\rangle_{\M}|\leq C
\left\|\sqrt{|g|}u \right\|_{W_{s',\delta'}^p(\M)}
\left\|\varphi \right\|_{W_{-s',-\delta'}^{p'}(\M)}\\ &
\leq C
\left\|\left(\sqrt{|g|}-1\right)
\right\|_{W_{s,\delta}^p(\M)}\left\|u
\right\|_{W_{s',\delta'}^p(\M)}\left\|\varphi 
\right\|_{W_{-s',-\delta'}^{p'}(\M)}
\\ &
\leq C_g
\left\|u
\right\|_{W_{s',\delta'}^p(\M)}\left\|\varphi 
\right\|_{W_{-s',-\delta'}^{p'}(\M)}.
\end{split}
\end{equation} 
This also implies that the topological dual with respect to $\langle
\cdot,\cdot\rangle_{(\M,g)}$ is isomorphic to $W_{-s',-\delta'}^{p'}(\M)$ (see
Proposition \ref{prop:duality}). 
Thus we have proven:
\begin{prop}
\label{prop:9}
 Let $s\in(\frac{n}{p},\infty)\cap [1,\infty)$, $\delta\geq-\frac{n}{p}$, 
$s-2\leq s\leq s$,
$\delta'\in\setR$ and $(\M,g)$ be 
 an asymptotically flat manifold  of the class $W_{s,\delta}^p$. 
Then the
$L^2$ inner--product (\ref{eq:l-2}) extends to a 
continuous bilinear form 
\begin{math}
\langle\cdot,\cdot\rangle_{(\M,g)}: W_{s',\delta'}^p(\M)\otimes
W_{-s',-\delta'}^{p'}(\M)\to\setR
\end{math}
that satisfies the 
inequality
\begin{equation}
\label{eq:inner4}
| \langle u,v\rangle_{(\M,g)}|\leq C_g
\left\|u\right\|_{W_{s',\delta'}^p(\M)}\left\|v\right\|_{W_{-s',-\delta'}^{p'}
(\M) },
\end{equation} 
where the constant $C_g$ depends on the metric $g$. 
Furthermore, if $s'>0$, $u\in 
W_{s',\delta'}^p(\M)$ and 
$v\in \mathcal{S}$, then 
$\langle
u,v\rangle_{(M,g)}=(u,v)_{(L^2,g)}$.
\end{prop}

In a similar manner we can treat the $L^2$--bilinear form between two
smooth vector 
fields $X$ and $Y$, that is, 
\begin{equation}
\label{eq:bilinear:3}
\begin{split}
& (X,Y)_{(L^2,{g})} :=\int_{\M}\langle X,Y\rangle_{g}  d\mu_{{g}}\\ 
= &\sum_{i=1}^{N}
\int_{E_{r_i}}\phi^\ast_i\left((\tilde{\chi}_i
\sqrt{|g|}g^{ab}X_a)(\tilde{\chi_i} Y_b)\right)dx+\sum_{j=1}^{N_0}
\int_{B_j}\varTheta^\ast_j\left((\tilde{\alpha_j}
\sqrt{|g|}g^{ab}X_a)(\tilde{\alpha_j}
Y_b)\right)dx\\ = &\langle
\sqrt{|g|}g^{ab} X_a, Y_b\rangle_{\M}=:\langle X,Y \rangle_{(\M,g)},
\end{split}
\end{equation} 
Applying again Propositions  \ref{prop:4} and
\ref{prop:Moser}, we conclude:
\begin{prop}
\label{prop:17}
Let $(\M,g)$ be an asymptotically flat manifold  of the class
$W_{s,\delta}^p$,  $s\in
(\frac{n}{p},\infty)\cap [\frac{1}{2},\infty)$ and $\delta\geq-\frac{n}{p}$. 
Then
the $L^2$-bilinear form (\ref{eq:bilinear:3}) 
 has a
continuous extension to a form  $\langle X,Y\rangle_{(\M,g)}$ on
$W_{s-1,\delta+1}^p(\M)\otimes W_{1-s,-\delta-1}^{p'}(\M)$ that satisfies the
inequality 
\begin{equation}
 |\langle X,Y \rangle_{(\M,g)}|\leq C_g
\|X\|_{W_{s-1,\delta+1}^p(\M)}\|Y\|_{W_{1-s,-\delta-1}^{p'}(\M)}.
\end{equation} 
Furthermore, if {$s>1$},{ $X\in 
W_{s-1,\delta+1}^p(\M)$} and $Y\in \mathcal{S}$, 
then
$\langle X,Y\rangle_{(\M,g)}=(X,Y)_{(L^2,g)} $.
\end{prop}

If $\varphi$ has compact support in a certain chart, then by
integration by parts, we obtain
\begin{equation*}
 (\nabla u,\nabla \varphi)_{(L^2,g)}=\int \sqrt{\det g}g^{ab}\partial_a u
\partial_b \varphi dx=-\int \Delta_g u \varphi d\mu_g.  
\end{equation*}  
Note that $W_{0,\delta}^p\neq L_\delta^p$ (see Remark \ref{rem:2.7}), however,
if
$u\in W_{s,\delta}^p(\M)$ and $s\geq 1$, then  by Theorem \ref{thm:2} $\nabla
u\in L^p_{\delta+1}$. Therefore it follows from Propositions \ref{prop:dual} and
\ref{prop:dual2} that $(\nabla u,\nabla \varphi)_{(L^2,g)}= \langle \nabla u,
\nabla \varphi\rangle_{(\M,g)}$, whenever $\varphi$ is
smooth.
This justifies the following definition.
\begin{defn}[Weak solutions]
\label{defn:weak solutions}
  Let $a_0, f\in W_{s-2,\delta+2}^p(\M)$ and $s\geq  1$.   A distribution $u\in
W_{s,\delta}^p(\M)$ is a weak solution of the
equation
\begin{equation}
\label{eq:Weighted_fractional_p:1}
 -\Delta_g u +a_0 u=f,
\end{equation}  
if
\begin{equation}
\label{eq:Weighted_fractional_p:2}
 (\nabla u,\nabla \varphi)_{(L^2,g)}
+\langle a_0 u,\varphi\rangle_{(\M,g)}=\langle 
f,\varphi\rangle_{(\M,g)}\quad \text{for all} \ \varphi\in
C_0^\infty(\M).
\end{equation}
 In the case that
  (\ref{eq:Weighted_fractional_p:1}) were an inequality, then the
  equality in (\ref{eq:Weighted_fractional_p:2}) would be replaced by
  the corresponding inequality and the test functions would be 
    non--negative.

\end{defn}

Next  we prove the weak maximum principle for the operator $-\Delta_g+a_0$
when $a_0\geq0$. For $p=2$ it was
proven by Maxwell \cite{maxwell06:_rough_einst}, {and on compact 
manifolds in
the $W^p_s$-spaces  by Holst, Nagy and Tsogtgerel
\cite{Holst_Nagy_Tsogtgerel}.} 
We recall that the distribution $a_0\geq 0$ if and only if
$\langle
a_0,\varphi\rangle_{(\M,g)}\geq 0$
for all non--negative $\varphi\in C_0^\infty(\M)$.
\begin{lem}
\label{lem:5}
 Assume  $(\M,g)$ is an asymptotically flat manifold  of the class
$W_{s,\delta}^p$, $a_0\geq 0$, $a_0\in W_{s-2,\delta+2}^p$, $s\in
(\frac{n}{p},\infty)\cap [1,\infty)$ and $
\delta> -\frac{n}{p}$.
If $u\in  W_{s,\delta}^p(\M)$ satisfies
\begin{equation}
\label{eq:elliptic18}
 -\Delta_g u+a_0u\leq 0,
\end{equation}
then $u\leq 0$ in $\M$.
\end{lem}

In order to prove it we need a pointwise multiplication in $W_s^p$ with
different
values of $p$. Such  properties were established in \cite[\S
4.4]{runst96:_sobol_spaces_fract_order_nemyt}, but for our needs it suffices to
use  Zolesio's formulation and his result~\cite{zolesio_77}.
\begin{prop}[Zolesio]
\label{prop:7}
 Let $0\leq s\leq\min\{s_1,s_2\}$, and $1\leq p_i,p<\infty$ be real
numbers satisfying
\begin{equation*}
 s_i-s\geq n\left(\frac{1}{p_i}-\frac{1}{p}\right) \quad \text{and}\quad 
s_1+s_2-s>n\left(\frac{1}{p_1}+\frac{1}{p_2}-\frac{1}{p}\right).
\end{equation*} 
Then the pointwise multiplication 
\begin{math}
 W_{s_1}^{p_1}(\setR^n) \times W_{s_2}^{p_2}(\setR^n) \to W_{s}^{p}(\setR^n)
\end{math} 
is continuous. 
\end{prop}
 
We shall also need the following known embedding (see e.g. \cite[Theorem
6.5.1]{berg_lofstrom_76}). 
\begin{prop}
 \label{prop:8}
If $s-\frac{n}{p}\geq s_0-\frac{n}{p_0}$ and $p\leq p_0$, then the embedding 
$W_s^p(\setR^n) \to W_{s_0}^{p_0}(\setR^n)$ is continuous.
\end{prop}
\begin{rem}
 We shall use Propositions \ref{prop:7} and \ref{prop:8} on an open bounded set 
$\Omega_0\subset \M$. The extension of them to bounded subsets of manifolds can
be 
carried out by standard methods. To see this  
let $\{(V_j,\varTheta_j)\}$ be a finite 
collection of charts such that $\Omega_0\subset \cup_j V_j$ and let 
$\{\alpha_j\}$ be partition of unity subordinate to $\{V_j\}$. Then
\begin{equation*}
 \|u\|_{W_s^p(\Omega_0)}=\sum_j\|\varTheta^\ast(\alpha_ju)\|_{W_s^p(\setR^n)}
\end{equation*}
is a norm, and   applying Propositions \ref{prop:7} and \ref{prop:8} to each 
of the norms in $\setR^n$  yields the results on subsets of a manifold. 
\end{rem}

\textit{Proof of Lemma \ref{lem:5}.}  We will show that $u\leq
\epsilon$ for an arbitrary positive $\epsilon$. Since $\delta>-\frac{n}{p}$, 
$u$ tends to zero at each end of $\M$ by
Proposition \ref{prop:2}(c). Hence  $\{u>\epsilon\}$ is a bounded set in $\M$. 
Let $w:=\max\{u-\epsilon,0\}$ and $\Omega_0\subset\M$ be an open set such that 
$\supp(w)\Subset \Omega_0$.  We recall that if a certain function, say $v$, has
support in the closure of $\Omega_0$, then
$\left\|v\right\|_{W_{s,\delta}^p(\M)}\simeq
\left\|v\right\|_{W_{s}^p(\Omega_0)}$.  Because of the limitations of
the embeddings of Proposition \ref{prop:8}, we split the proof into
two cases, $p\geq 2$ and $p\leq 2$.

Starting with $p\geq 2$, we have that $W_{1}^p(\Omega_0)\subset
W_{1}^2(\Omega_0)$. Hence $w:=\max\{u-\epsilon,0\}$ also belongs to
$W_{1}^2(\Omega_0)$. We now claim that  $uw\in W_{2-s}^{p'}(\Omega_0)$, but
since  $2-s\leq 1$,  it suffices to show that $uw\in
W_1^{p'}(\Omega_0)$. Applying Proposition \ref{prop:7}, we have that 
\begin{equation*}
 \left\|uw\right\|_{W_{1}^{p'}(\Omega_0)}\lesssim
\left\|u\right\|_{W_{s}^{p}(\Omega_0)}\left\|w\right\|_{W_{1}^{p'}(\Omega_0)},
\end{equation*}
and since $p'\leq 2$, we have by the H\"older inequality that 
\begin{equation*}
 \left\|w\right\|_{ W_{1}^{p^\prime}(\Omega_0)} 
\lesssim
\left({\rm Vol}(\Omega_0,g)\right)^{\frac{p-2}{2p(p-1)}}\left\|w\right\|_{
W_{1}^{2}(\Omega_0)}^{\frac{p'}{p}}.
\end{equation*}
Hence  $uw$ belongs  $W_{2-s}^{p'}(\Omega_0)$, and   since
${a_0}_{\mid_{\Omega_0}}\in W_{s-2}^p(\Omega_0)$ the bilinear form 
$\langle a_0 ,uw\rangle_{(\M,g)}$ is finite.  Moreover $uw\geq0$, so     by 
the density property of Besov spaces  $\langle a_0,uw\rangle_{(\M,g)}\geq0$. 
Combining these with 
(\ref{eq:elliptic18}) and
(\ref{eq:Weighted_fractional_p:2}), we obtain that
\begin{equation*}
 0\leq \langle a_0 ,uw\rangle_{(\M,g)}=\langle a_0 u,w\rangle_{(\M,g)}
\leq -(\nabla u,\nabla w)_{(L^2,g)}=-(\nabla w,\nabla w)_{(L^2,g)}
\lesssim -\left\|\nabla w\right\|_{L^2(\Omega_0)}^2.
\end{equation*}
Thus $w$ is a constant in $\Omega_0$ and since it vanishes on the boundary,  it 
is identically zero in $\Omega_0$. Consequently 
$u\leq\epsilon$, and  that completes the proof when $p\geq 2$.

In the case of $p\leq 2$, we first claim that $a_0\in W_{-1}^n(\Omega_0)$.
To see this we apply Proposition \ref{prop:8} with $s_0=-1$ and 
    $p_0=n$, then the inequality $s-2-\frac{n}{p}\geq -1-\frac{n}{n}$ holds,
since
    $s\geq \frac{n}{p}$. The requirement $p\leq p_0=n$ holds since throughout
the paper  $n\geq 2$. So we conclude $a_0\in W_{-1}^n(\Omega_0)$.

Applying again Proposition \ref{prop:8} we have that $u\in W_1^n(\Omega_0)$,
hence   $w=\max\{u-\epsilon,0\}$ belongs to $w\in W_1^n(\Omega_0)$, and 
Proposition \ref{prop:7} yields that
 \begin{equation*}
      \left\|uw\right\|_{W_{1}^{n'}(\Omega_0)}\lesssim  
\left\|u\right\|_{W_{s}^{p}(\Omega_0)}\left\|w\right\|_{W_{1}^{n}(\Omega_0)}.
    \end{equation*}
Therefore by (\ref{eq:3}) the bilinear form $\langle
a_0,uw\rangle_{(\M,g)}$ is well 
defined. We now  complete the proof as in the
case $p\geq2$.
\hfill{$\square$}

Let $e $ be the metric such that in any local coordinates
$e_{ab}=\delta_{ab}$. In case $\M$ has one end, then $e$ is the Euclidean 
metric. 
For $t\in[0,1]$  we consider a family of metrics $tg+(1-t)e$.  
Then $\mathcal{M}$ equipped with this  metric is asymptotically flat of class 
$W_{s,\delta}^p$, since 
$\left(tg+(1-t)e\right)_{ab}-\delta_{ab}=t(g_{ab}-\delta_{ab})\in 
W_{s,\delta}^p(E_{r_{i}})$.
 So by the weak
maximum principle,  the operator 
\begin{equation}
\label{eq:elliptic:4}
 \Delta_{\{tg+(1-t){e}\}}+t a_0: W_{s,\delta}^p(\M)\to
W_{s-2,\delta+2}^p(\M)
\end{equation}
is injective for all $t\in [0,1]$. 
Lemma \ref{lem:elliptic1} implies the for $t=0$ the operator 
(\ref{eq:elliptic:4}) is an isomorphism, and by Corollary \ref{cor:2} it has a 
closed range. In this situation we can apply    standard homotopy 
arguments (see 
e.g. \cite[Lemma 2]{cantor79_global}) and  obtain the following lemma: 
\begin{lem}
\label{lem:1}
Assume  $(\M,g)$ is an asymptotically flat manifold  of the class
$W_{s,\delta}^p$, $a_0\geq 0$,  $a_0\in W_{s-2,\delta+2}^p$,
 $s\in (\frac{n}{p},\infty)\cap [1,\infty)$ and
$\delta\in(-\frac{n}{p},-2+\frac{n}{p^\prime})$. Then for any $f\in
W_{s-2,\delta+2}^p(\M)$ equation
\begin{equation*}
 -\Delta_g u+a_0u=f
\end{equation*} 
has a unique solution $u$ satisfying 
\begin{equation}
\label{eq:elliptic7}
 \left\| u \right\|_{W_{s,\delta}^p(\M)}\leq C\left\| f
\right\|_{W_{s-2,\delta+2}^p(\M)},
\end{equation} 
where the constant $C$ is independent on $f$.
\end{lem}
   
 \section{Semi--linear elliptic equations}
\label{sec:semi-linear}

In this section we establish an existence and uniqueness theorem for a
semi--linear equation whose principal part is the Laplace--Beltrami
operator on an asymptotically flat Riemannian manifold. The method of
sub and super solutions is used frequently in such types of problems,
however, we will employ a homotopy argument similar to the one
presented by Cantor
\cite{cantor79}. The authors applied this method in \cite{BK3} for
$p=2$ and $s\geq2$, and here, beside extending it to the
$W_{s,\delta}^p$-spaces, we simplify some of the steps of the proof by
computing the norm by means of the bilinear form  (\ref{eq:dual3}). The 
conditions of
Theorem \ref{thm2} below could be relaxed to some extensions, but we
refrain dealing with it here. 

Let
\begin{equation*}
 F(u,x):=h_1(u)m_1(x)+\cdots+h_N(u)m_N(x),
\end{equation*} 
be a function, where $h_i:(-1,\infty)\to [0,\infty)$ is $C^1$ non--increasing
function, $m_i\geq0$ and $m_i\in W_{s-2,\delta+2}^p(\M)$. The typical example
of $h_i(t)$ is $(1+t)^{-\alpha_i}$ with $\alpha_i>0$. 
\begin{thm}
\label{thm2}
 Assume  $(\M,g)$ is an asymptotically flat manifold  of the class
$W_{s,\delta}^p$, $a_0\in W_{s-2,\delta+2}^p$,
$a_0\geq 0$, $s\in (\frac{n}{p},\infty)\cap [1,\infty)$ and $
\delta\in(-\frac{n}{p},-2+\frac{n}{p^\prime})$.
Then the equation
\begin{equation*}
 -\Delta_g u+a_0u=F(u,\cdot)
\end{equation*}
has a unique non--negative solution $u\in W_{s,\delta}^p(\M)$.
\end{thm}
\begin{proof}
 We define a map $\Phi:\left(W_{s,\delta}^p(\M)\cap
\left\{u>-1\right\}\right)\times [0,1]\to W_{s-2,\delta+2}^p(\M)$ by
\begin{equation}
\label{eq:elliptic23}
 \Phi(u,\tau)=-\Delta_g u+a_0u-\tau F(u,\cdot)
\end{equation} 
and set $J=\{\tau\in[0,1]: \Phi(u,\tau)=0\}$. Lemma \ref{lem:1} implies that
$0\in J$ and therefore it suffices to show that $J$ is an open and closed set.
Since the functions $h_i$ are non--increasing, $\frac{\partial F}{\partial
u}(u,\cdot)\leq0$, and therefore the operator
\begin{equation*}
 Lw:=\left(\frac{\partial \Phi}{\partial u}(u,\tau)\right)w=-\Delta_g
w+a_0w-\tau
\frac{\partial F}{\partial u}(u,\cdot)w.
\end{equation*}
satisfies the assumptions of Lemma \ref{lem:1}.   Hence $\frac{\partial
\Phi}{\partial u}$ is an isomorphism  and this implies that  $J$ is an open
set (see e.g. 
\cite[\S17.2]{gilbarg83:_ellip_partial_differ_equat_secon_order}). 
The essential difficulty is to show that $J$ is a closed set. So
let $u(\tau)$ be a
solution to $\Phi(u,\tau)=0$. We first claim that there
is a positive constant $C_0$ independent of $\tau$ such that
\begin{equation}
\label{eq:elliptic5}
 \left\|u(\tau) \right\|_{W_{s,\delta}^p(\M)}\leq C_0.
\end{equation}

Now   the weak maximum principle, Lemma \ref{lem:5}, implies
 that $u(\tau)\geq0$, and hence, $h_i(u(\tau))\leq h_i(0)$.  Therefore
 by Proposition \ref{prop:9}, 
for any nonnegative
$\varphi\in C_0^\infty(\M)$ we have  that
\begin{equation}
\label{eq:elliptic6}
0\leq \langle h_i \left(u(\tau)\right)m_i,\varphi\rangle_{(\M,g)}
\leq \langle h_i \left(0\right)m_i,\varphi\rangle_{(\M,g)} \leq C_g
h_i(0)\left\| m_i\right\|_{W_{s-2,\delta+2}^p(\M)}\left\|
\varphi\right\|_{W_{2-s,-\delta-2}^{p'}(\M)}.
\end{equation}
Hence by formulas (\ref{eq:bilinear:1}), (\ref{eq:bilinear:4}) and inequality
(\ref{eq:dual3})
\begin{equation}
\label{eq:elliptic9}
 \|h_i\left(u(\tau)\right)m_i\|_{W_{s-2,\delta+2}^p(\M)}\leq C_g h_i(0)
\|m_i\|_{W_{s-2,\delta+2}^p(\M)}
\end{equation}
for $i=1,\ldots,N$.  By Lemma \ref{lem:1}, and inequalities (\ref{eq:elliptic7})
and  (\ref{eq:elliptic9})  we obtain that
\begin{equation*}
\left\| u(\tau)\right\|_{W_{s,\delta}^p(\M)}\leq C
\left\| F(u(\tau),\cdot)\right\|_{W_{s-2,\delta+2}^p(\M)}\leq C C_g
\sum_{i=1}^N h_i(0)\left\| m_i\right\|_{W_{s-2,\delta+2}^p(\M)}=:C_0,
\end{equation*}
and that proves  (\ref{eq:elliptic5}).

Differentiating (\ref{eq:elliptic23}) with respect to $\tau$ gives
\begin{equation}
\label{eq:elliptic8}
 -\Delta_g u_\tau+a_0 u_\tau-\tau\frac{\partial F}{\partial u}(u(\tau),
\cdot)u_\tau=F(u(\tau),\cdot),
\end{equation}
 where $u_\tau$ denotes the derivative of $u(\tau)$ with respect to $\tau$.
By Propositions \ref{prop:4} and \ref{prop:Moser}, both
$\left\|F(u(\tau),\cdot)\right\|_{W_{s-2,\delta+2}^p(\M)}$ and 
$\left\|\frac{\partial F}{\partial
u}(u(\tau),\cdot)\right\|_{W_{s-2,\delta+2}^p(\M)}$ are bounded by
$\left\|u(\tau)\right\|_{W_{s,\delta}^p(\M)}$. In addition,  $\frac{\partial
F}{\partial u}\leq 0$, thus the operator (\ref{eq:elliptic8}) satisfies the
conditions of Lemma \ref{lem:1}, and hence it possesses  a solution $u_\tau$
in $W_{s,\delta}^p(\M)$. 
  
We now show that $\left\|u_\tau\right\|_{W_{s,\delta}^p(\M)}$ is bounded by a
constant independent of $\tau$. By Lemma \ref{lem:1}, equation 
\begin{equation*}
 -\Delta_g w+a_0w=F(u(\tau),\cdot).
\end{equation*}
has a solution $w$ that satisfies the inequality
$\left\|w\right\|_{W_{s,\delta+}^p(\M)}\leq C
\left\|F(u(\tau),\cdot)\right\|_{W_{s-2,\delta+2}(\M)}$.  Since the bound 
of $\left\|F(u(\tau),\cdot)\right\|_{W_{s-2,\delta+2}(\M)}$ is independent of 
$\tau$ by  (\ref{eq:elliptic5}), we conclude that 
\begin{equation*}
 \left\|w\right\|_{W_{s,\delta}^p(\M)}\leq K
\end{equation*}
and the constant $K$ is independent of $\tau$. From the weak maximum principle,
Lemma \ref{lem:5}, we get that  $u_\tau\geq0$ and hence 
$(-\Delta_g +a_0)(w-u_\tau)=-\tau\frac{\partial
F}{\partial u}(u(\tau),\cdot)u_\tau\geq 0$.
Thus $(w-u_\tau)\geq 0$, again by
the maximum principle, and by  
Proposition \ref{prop:9} we have that
\begin{equation*}
 0\leq \langle u_\tau,\varphi\rangle_{(\M,g)}\leq  
 \langle w,\varphi\rangle_{(\M,g)}\leq C_g
\left\| w\right\|_{W_{s,\delta}^p(\M)}\left\|
\varphi\right\|_{W_{-s,-\delta}^{p'}(\M)}\leq C_g K
\left\|\varphi\right\|_{W_{-s,-\delta}^{p'}(\M)}  
\end{equation*}
 for any non--negative $\varphi\in C_0^\infty(\M)$. Thus using again  the dual 
estimate of the norm (\ref{eq:dual3}),  we have obtained that
\begin{equation*}
 \left\|u_\tau\right\|_{W_{s,\delta}^p(\M)}\leq C_g K,
\end{equation*}

This implies  that the norm $ \left\| u(\tau)\right\|_{W_{s,\delta}^p(\M)}$
is a
Lipschitz function   of $\tau$, that is, 
$|\left\|u(\tau_1)\right\|_{W_{s,\delta}^p(\M)}-\left\|u(\tau_2)\right\|_{W_{s,
\delta}^p(\M)}|\leq C_gK|\tau_1-\tau_2|$.
Therefore if $\{\tau_k\}\subset J$ and
$\tau_k\to \tau_0$,  then $\{u(\tau_k)\}$ is a Cauchy sequence in
$W_{s,\delta}^p(\M)$
and hence $J$ is a closed set. That completes the proof of the existence.

As for  the uniqueness, assume there are two different solutions
$u_1$ and $u_2$.
Then at least one of the sets $\Omega_+:=\{x\in \M: u_1(x)-u_2(x)>0\}$ or
$\Omega_-:=\{x\in \M: u_1(x)-u_2(x)>0\}$ is non--empty.
Suppose that $\Omega_+\neq\emptyset$, then $w=u_1-u_2$ satisfies the
equation $-\Delta_g w+a_0w=F(u_1)-F(u_2)\leq 0 $ in $\Omega_+$, since
$F(u,x)$ is non--increasing as a function of $u$.
Now the weak maximum principle, Lemma \ref{lem:5}, implies that $w\leq
0$ in $\Omega_+$ and hence it must be an empty set.
A similar contradiction occurs if $\Omega_-\neq\emptyset$.

\end{proof}

\section{The Brill--Cantor criterion}
\label{sec:brill-cantor-theorem}

 Let  $(\M,g)$ be an asymptotically flat manifold  of  class
$W_{s,\delta}^p$ and $R(g)$ be the scalar curvature. Throughout this section
$n\geq 3$. We set
$2^*=\frac{2n}{n-2}$ and $s_n=\frac{n-2}{4(n-1)}$.  
 Following
Choquet--Bruhat, Isenberg, and  York   \cite{y.00:_einst_euclid} and Maxwell
\cite{maxwell05:_solut_einst}, we define. 
\begin{defn}
\label{def:Yamabe}
An asymptotically flat manifold $(\M,g)$ is in the positive Yamabe class if 
\begin{equation}
\label{eq:brill-cantor}
 \inf_{\varphi\in C_0^\infty(\M)}\dfrac{(\nabla \varphi,\nabla\varphi)_{(L^2,g)}
+s_n\langle R(g),
\varphi^2\rangle_{(\M,g)}}{\left\|\varphi\right\|_{L^{2^*}}^2}>0.
\end{equation} 
\end{defn}
This condition is conformal invariant under the scaling
$g'=\phi^{\frac{4}{n-2}}g$ \cite{y.00:_einst_euclid}. For $s\geq 2$ the metric
$g\in W_{2,\delta}^p(\M)$ and by Theorem \ref{thm:2}, $R(g)\in L^p(\M)$. So in
this case formula (\ref{eq:bilinear:4}) implies that $\langle
R(g),\varphi^2\rangle_{(\M,g)}=(R(g),\varphi^2)_{(L^2,g)}$ and 
condition (\ref{eq:brill-cantor}) takes the common form
 \begin{equation*}
 \inf_{\varphi\in C_0^\infty(\M)}\dfrac{\int_\M\left((\nabla
\varphi,\nabla\varphi)_g+s_n  R(g)
\varphi^2\right)d\mu_g}{\left\|\varphi\right\|_{L^{2^*}}^2}>0.
\end{equation*} 

Though condition (\ref{eq:brill-cantor}) is similar to the Yamabe functional on
compact manifolds (\cite[Ch. 5]{Aubin_1998}, \cite[Ch. 7]{Choquet_2009}), it has
a different interpretation on asymptotically flat
manifolds, namely, in that case being  in the positive Yamabe class is 
equivalent to  the existence of a conformal flat metric. 
\begin{thm}
\label{thm:brill-cantor}
 Let  $(\M,g)$ be an asymptotically flat manifold  of the class
$W_{s,\delta}^p$, and assume that   $s\in
(\frac{n}{p},\infty)\cap [1,\infty)$ and   $\delta\in
(-\frac{n}{p},-2+\frac{n}{p^\prime})$. Then $(\M,g)$ is in the positive
Yamabe class if and only if there is a conformally equivalent metric $g'$ such
the $R(g')=0$. 
\end{thm}

This type of result was first proven  in \cite{cantor81:_laplac} for
$s>\frac{n}{p}+2$ and $1<p<\frac{2n}{n-2}$.  Since then the 
regularity assumptions  were improved by several
authors
\cite{Choquet_2009, Choquet_Isenberg_Pollack, y.00:_einst_euclid,
  maxwell05:_solut_einst}, however, they dealt only with Sobolev spaces of
integer order, and when  $s=2$ they are restricted to
$p>\frac{n}{2}$. For $p=2$ it was  proven for
all $s>\frac{n}{2}$ in \cite{maxwell06:_rough_einst}.
Thus Theorem \ref{thm:brill-cantor}
improves regularity and extends the range of $p$ to $(1,\infty)$.

\textit{Proof of Theorem \ref{thm:brill-cantor}.} We prove only that condition
(\ref{eq:brill-cantor}) implies the existence of a
flat metric. The converse assertion requires no special attention of the
weighted Besov spaces and we refer to \cite{Choquet_Isenberg_Pollack,
y.00:_einst_euclid} for this part.

 We consider the following  conformal transformation $g'=\phi^{\frac{4}{n-2}}g$.
It is known that the metric $g'$ has scalar curvature zero if and only if
equation 
(see e.g. \cite{Aubin_1998})
\begin{equation}
\label{eq:yam2}
 -\Delta_g\phi+s_nR(g)\phi=0,  
\end{equation} 
possesses a positive solution $\phi$ such  that $\phi-1\in W_{s,\delta}^p(\M)$.
Setting $u=\phi-1$, then (\ref{eq:yam2}) becomes
\begin{equation}
\label{eq:yam3}
 -\Delta_gu+s_nR(g)u=-s_n R(g).  
\end{equation} 
In order to assure that equation (\ref{eq:yam3}) has a solution, 
it suffices to show that  the operator $-\Delta_g+\tau s_nR(g)$
has a trivial kernel for each
$\tau\in [0,1]$. The crucial point is to  estimate  the numerator  of
(\ref{eq:brill-cantor}) in terms of the $W_{s,\delta}^p(\M)$-norm. Starting
with the second term, we have by Proposition \ref{prop:9} that
\begin{equation}
\label{eq:yam8}
 |\langle R(g),\varphi^2\rangle_{(\M,g)}|=|\langle
R(g)\varphi^2,1\rangle_{(\M,g)}|\lesssim
\left\|R(g)\varphi^2\right\|_{W_{s-2,\delta''}^p(\M)}\left\| 1\right\|_{W_{2-s,
-\delta''}^{p'}(\M)}.
\end{equation} 
Obviously, $\left\|
1\right\|_{W_{2-s, -\delta''}^{p'}(\M)}\leq \left\|
1\right\|_{W_{1, -\delta''}^{p'}(\M)}$ and the last term  is finite if
$\delta''>\frac{n}{p'}$.
Take now  $\delta'$ satisfying the condition
\begin{equation}
 \label{eq:yam7}
\frac{n}{p'}<\delta''\leq \delta+2+2\delta'+\frac{2n}{p},
\end{equation} 
 and  Proposition \ref{prop:11}, with $\delta_1=\delta+2$ and
$\delta_2=\delta_3=\delta'$, then we obtain that
\begin{equation}
\label{eq:yamb11}
 \left\|R(g)\varphi^2\right\|_{W_{s-2,\delta''}^p(\M)}\lesssim
\left\|R(g)\right\|_{W_{s-2,\delta+2}^p(\M)}
\left(\left\|\varphi\right\|_{W_{s,\delta'}^p(\M)}\right)^2.
\end{equation} 

For the first  term of the numerator  of
(\ref{eq:brill-cantor}), we  have by the identity (\ref{eq:bilinear:4})  that  
 \begin{equation*}
 (\nabla\varphi,\nabla\varphi)_{(L^2,g)} =(|\nabla
\varphi|^2_{{g}},1)_{(L^2,{g})}=\langle \sqrt{|g|}g^{ab}\partial_a
\varphi\partial_b\varphi,1\rangle_\M.
\end{equation*}
So by inequality (\ref{eq:bilinear:2}),
\begin{equation}
\label{eq:yam6}
|(\nabla\varphi,\nabla\varphi)_{(L^2,g)}|\
\lesssim
\left\|\sqrt{|g|}g^{ab}\partial_a
\varphi\partial_b\varphi\right\|_{W_{s-1,\delta''}^p(\M)}
\left\|1\right\|_{W_{1-s,-\delta''}^{p'}(\M)}.
\end{equation} 
As in the previous term, $\left\|1\right\|_{W_{1-s,-\delta''}^{p'}(\M)}$ is
finite if $\delta''>\frac{n}{p}$.
Writing 
\begin{equation*}
 \sqrt{|g|}g^{ab}\partial_a
\varphi\partial_b\varphi=\left(\sqrt{|g|}g^{ab}-\delta^{ab}\right)\partial_a
\varphi\partial_b\varphi+ \partial^a\varphi\partial_a\varphi
\end{equation*}
and assuming that $\delta'$ satisfies (\ref{eq:yam7}),
we then  can apply again Propositions
\ref{prop:11}, with $\delta_1=\delta$ and
$\delta_2=\delta_3=\delta'+1$, Proposition \ref{prop:Moser}, and get
that
\begin{equation}
\label{eq:yamb12}
 \left\|\left(\sqrt{|g|}g^{ab}-\delta^{ab}\right)\partial_a
\varphi\partial_b\varphi\right\|_{W_{s-1,\delta''}
^p(\M)}\lesssim  \left\|g-1\right\|_{W_{s,\delta}
^p(\M)} \left( \left\|
|\nabla\varphi|\right\|_{W_{s-1,\delta'+1}^p(\M)} \right)^2.
\end{equation}
By Proposition \ref{prop:4},
\begin{equation}
\label{eq:yamb13}
 \left\| |\nabla\varphi|^2\right\|_{W_{s-1,\delta''}
^p(\M)}\lesssim  \left( 
\left\||\nabla\varphi|\right\|_{W_{s-1,\delta'+1}^p(\M)}\right)^2
\end{equation}
whenever $\delta'$ also  satisfies the condition
\begin{equation}
 \label{eq:yam10}
\frac{n}{p'}<\delta''\leq2(\delta'+1)+\frac{n}{p}.
\end{equation}

We are now in a position to show that if $(\M,g)$ is in the positive Yamabe
class, then $-\Delta_g+\tau s_nR(g)$ is an injective operator. For
$\tau=0$ it is
injective by the weak maximum principle. For
each $\tau\in (0,1]$, we assume the contrary, that is, there is $0\not\equiv
u\in
W_{s,\delta}^p(\M)$ such that $-\Delta_g u+\tau s_nR(g)u=0$. Then by
Proposition
\ref{prop:10}, $u\in W_{s,\delta'}^p(\M)$ for any $\delta'\in
(-\frac{n}{p},-2+\frac{n}{p'})$. We can always choose $\delta'\in
(-\frac{n}{p},-2+\frac{n}{p'})$ so that both 
(\ref{eq:yam7}) and (\ref{eq:yam10}) hold for any given
$\delta$ in $(-\frac{n}{p},-2+\frac{n}{p'})$. Choosing such $\delta'$ and
taking  a sequence   $\{\varphi_k\}\subset C_0^\infty(\M)$ such that
$\varphi_k\to u$ in $W_{s,\delta'}^p(\M)$, 
 then   inequalities (\ref{eq:yam8}), (\ref{eq:yamb11}),
(\ref{eq:yam6}),(\ref{eq:yamb12}) and (\ref{eq:yamb13})  
imply that the numerator of  (\ref{eq:brill-cantor}) is bounded by  
$\left\|\varphi_k\right\|_{W_{s,\delta'}^p(\M)}$.  Hence we may pass to the
limit and obtain that
\begin{equation*}
\begin{split}
 0 & = (\nabla u,\nabla \varphi_k)_{L^2(\M,g)}+\tau s_n\langle
R(g)u, \varphi_k\rangle_{(\M,g)}
\\ &=   \lim_k \left((\nabla \varphi_k,\nabla \varphi_k)_{L^2(\M,g)}+\tau
s_n\langle
R(g), \varphi_k^2\rangle_{(\M,g)}\right)\\ & \geq\tau \lim_k \left((\nabla
\varphi_k,\nabla \varphi_k)_{L^2(\M,g)}+ s_n\langle
R(g), \varphi_k^2\rangle_{(\M,g)}\right).
\end{split}
\end{equation*} 
Since $(\M,g)$ is in the positive Yamabe class, the last term of the above
inequalities is positive and obviously this is a  contradiction.

Having shown that $-\Delta_g+\tau s_nR(g)$ is injective, we conclude by
Corollary \ref{cor:2} and the homotopy argument as in Lemma \ref{lem:1}
  that equation
(\ref{eq:yam3}) has a unique solution. Let 
$u$ be the solution and set $\phi=1+u$, then it remains to show that $\phi>0$. 
We follow here   \cite{cantor81:_laplac, maxwell05:_solut_einst}.  Let
$u_\lambda$ be a solution to
$-\Delta_g
u_\lambda+\lambda s_n
R(g) u_\lambda=-\lambda s_n R(g)$ and set $J=\{\lambda\in [0,1]:
\phi_\lambda(x)=1+u_\lambda(x)>0\}$. By Lemma \ref{lem:elliptic2} and  its
version on manifolds  Lemma \ref{lemma:3.5}, there is a constant $C$ 
independent of 
$\lambda$ and $\delta'<\delta$  such that
\begin{equation*}
 \left\|u_\lambda\right\|_{W_{s,\delta}^p(\M)}\leq C\left\{
\left\|\lambda s_n 
R(g)\right\|_{W_{s-2,\delta+2}^p(\M)}+\left\|u_\lambda\right\|_{W_{s-1,\delta'}
^p(\M)}\right\}.
\end{equation*} 
Now it follows from the compact embedding, Proposition
\ref{prop:2} (b) and the above estimate that   
$\left\|u_\lambda\right\|_{W_{s,\delta}^p(\M)}$ is  a continuous
function of $\lambda$. Hence, by the embedding into the continuous,
Proposition \ref{prop:2}(c), $\phi_\lambda-1$ is continuous  in $C_\beta^0$ for
some $\beta>0$ with respect to $\lambda$.
Thus $J$ is a open and non--empty set,
since $0\in J$. So if $J\not=[0,1]$, then there exists a $0<\lambda_0<1$ such
that
$\phi_{\lambda_0}\geq 0$. Then by the Harnack inequality  $\phi_{\lambda_0}>0$
and consequently $\phi_1=\phi>0$. For
details how to apply the Harnack inequality under the present regularity
assumption
see \cite[Lemma 35]{Holst_Nagy_Tsogtgerel} and \cite[lemma
5.3]{maxwell06:_rough_einst}. 
\hfill{$\square$}

\section{Applications to the Constraint Equations of the Einstein--Euler
Systems} 
\label{sec:appl-constr-equat} 
 
In this section we  describe briefly the initial data for the 
Einstein--Euler system, for more details we refer to \cite{BK3,BK8}. 
In \cite{BK3} we constructed the initial data in the Hilbert space 
$W_{s,\delta}^2(\M)$ and here we apply the results of the previous 
sections in order to construct the initial data in the weighted Besov spaces 
$W_{s,\delta}^p(\M)$ for $1<p<\infty$. 
 
The Einstein--Euler system describes a relativistic self--gravitating perfect
fluid. The fluid quantities are the energy density $\rho$, the pressure $p$ and
a unite time--like velocity vector $u^\alpha$. In this section Greek indexes
take the values $0,1,2,3$. The evolution of the gravitational
fields is described by the Einstein equations
\begin{equation*} 
  R_{\alpha\beta}-\frac{1}{2} \   g_{\alpha\beta}R = 8\pi T_{\alpha\beta}, 
\end{equation*} 
where $g_{\alpha\beta}$ is a semi Riemannian metric with signature
$(-,+,+,+)$, $R_{\alpha\beta}$ is the Ricci curvature tensor and
$T_{\alpha\beta}$ is the energy--momentum tensor of the matter, which
in the case of a perfect fluid  takes the form
\begin{equation}
\label{eq:eineul:2} 
  T^{\alpha\beta} = (\rho+p) u^\alpha u^\beta + 
  p\ g^{\alpha\beta}. 
\end{equation} 
The evolution of the fluid 
is described by the Euler equations $\nabla_\alpha T^{\alpha\beta}=0$. This
system contains more unknowns than equations and
therefore an additional relation is indispensable. 
The usual strategy is to
introduce an equation of state, which connects $p$ and $\rho$. Here we 
consider the analogue of the non--relativistic polytropic equation of 
state and that is  given by 
\begin{equation} 
  \label{eq:ee8} 
  p = p(\rho) = \kappa\rho^{\gamma}, \qquad 
  1<\gamma, \qquad \kappa\in \setR^{+}. 
\end{equation} 
 
In the context of astrophysics, isolated systems 
cannot  have a density that is bounded below by
a positive constant.
It either falls off at infinity, or has compact support.
That causes the corresponding symmetric hyperbolic system to
degenerate (see \cite{BK8} for details).
Following Makino
\cite{makino86:_local_exist_theor_evolut_equat_gaseous_stars}, we
regularize the symmetric hyperbolic system by the variable change
\begin{equation} 
\label{eq:ee6} 
 w=\rho^{\frac{\gamma-1}{2}}. 
\end{equation}

The initial data of the Einstein--Euler system are a proper Riemannian
metric $g$ and a symmetric $(2,0)$--tensor $K_{ab}$, given on a three
dimensional manifold $\M$.
The matter variables are $(z,j^a)$, where $z$ denotes the energy
density and $j^a$ the momentum density, and in addition, there are
initial data for the fluid.
These are the Makino variable $w$ and the velocity vector $u^\alpha$.
The data must satisfy the constraint equations
\begin{equation} 
    \label{eq:EE1} 
    \left\{ 
      \begin{array}{cccl} 
        R(g) - K_{ab}K^{ab}+(g^{ab}K_{ab})^2 &=&16\pi z& \text{Hamiltonian constraint}\\ 
        {}^{(3)}\nabla_b K^{ab}-  
        \nabla^b(g^{bc}K_{bc}) &=& -8\pi j^a& \text{Momentum
constraint} 
      \end{array}\right.. 
  \end{equation}  
 
Let $\tilde{u}^\alpha$ denote the projection of the velocity vector $u^\alpha$ 
on the initial manifold $\M$. The projections of the 
energy--momentum tensor $T_{\alpha\beta}$ twice
on the unit normal $n^\alpha$, once on $n^\alpha$ and once on $\M$, lead to
the following relations 
\begin{equation} 
  \label{eq:ee7} 
  \left\{ 
    \begin{array}{ccc} 
      z &= &\rho+(\rho+p)g_{ab}\tilde{u}^a\tilde{u}^b\\ 
      j^\alpha 
      &=&(\rho+p)\tilde{u}^a\sqrt{1+ g_{ab}\tilde{u}^a\tilde{u}^b} 
    \end{array}\right.. 
\end{equation}

We use the well--known conformal method for solving the constraint 
equations (\ref{eq:EE1}). This method starts by giving some free quantities 
 and the solutions of the constraints are obtained in the end by 
rescaling these with appropriate powers of a scalar function $\phi$. This
function is the solution of the Lichnerowicz equation (\ref{eq:ee5}). In the
case of the fluid the quantities
which can be rescaled in 
a way which is consistent with the general scheme are $z$ and $j^a$,  
and not the quantities $w$ and $\tilde u^a$. Therefore, in order to provide
initial data
for the fluid variables  $(w,\tilde{u}^a)$, equations (\ref{eq:ee7}) must be 
inverted.

Taking into account the variable change (\ref{eq:ee6}) and the equation of 
state (\ref{eq:ee8}), then (\ref{eq:ee7})  is equivalent to the inversion of 
the map (see \cite[\S4]{BK3} for details) 
\begin{equation} 
\label{eq:EE2} 
\begin{split} 
 \Phi\left(w,\tilde{u}^a\right)&:=\left(w\left\{ 
1+\left(1+\kappa 
w^2\right)\left(g_{ab}\tilde{u}^a\tilde{u}^b\right)\right\}^{\frac{\gamma-1}{2}}
,
\dfrac{\left(1+\kappa 
w^2\right)\tilde{u}^a\sqrt{1+g_{ab}\tilde{u}^a\tilde{u}^b}}{1+\left(1+\kappa 
w^2\right)\left(g_{ab}\tilde{u}^a\tilde{u}^b\right)}\right)\\ &= 
(z^{\frac{\gamma-1}{2}},\frac{j^a}{z}). 
\end{split} 
\end{equation}  
 
The inversion of this map under certain condition was established in
\cite{BK3}. 

\begin{thm}[Reconstruction theorem for the initial data] 
 \label{thm4} 
Let $g$ be a Riemannian metric, then there is a continuous function $S:[0,1)\to 
\setR$ such that if  
\begin{equation} 
\label{eq:ee3} 
 0\leq z^{\frac{\gamma-1}{2}}\leq S\left(z^{-1} \sqrt{ g_{ab} j^a j^b}\right), 
\end{equation}   
then system (\ref{eq:EE2}) has a unique inverse. 
 
\end{thm} 
 
 Since Condition (\ref{eq:ee3}) is not invariant under
  scaling, the unscaled initial data for the energy and momentum densities must
  satisfy it.

Therefore there are two types of free data, the geometric data $(\bar
g,\bar A^{ab})$ where $\bar g $ is a Riemannian metric, $\bar A^{ab}$
is divergence and trace free form, and the matter data
$(\hat{z}^{\frac{\gamma-1}{2}},\hat j^a)$,  which are
  constructed using (\ref{eq:ee7}) but with the flat metric $\hat g$.

We also assume that $(\M,\bar g)$ belongs in the positive Yamabe class
(see Definition \ref{def:Yamabe}) and has no Killing vector fields in
$W_{s,\delta}^p(\M)$ (for $p=2$ and $s>\frac{3}{2}$ this assumption
was verified in \cite{maxwell06:_rough_einst}).
 
\begin{thm}[Solution of the constraint equations] 
  \label{thm3} 
Let $\M$ be  a Riemannian manifold and   $(\bar g,\bar 
  A^{ab},\hat{z}^{\frac{\gamma-1}{2}},\hat j^a)$ be free data  such that
$(\M,\bar g)$ is asymptotically flat of the 
class $W_{s,\delta}^p$ and belongs to the positive Yamabe class, 
$\bar{A}^{ab}\in W_{s-1,\delta+1}^p(\M)$, $(\hat{z}^{\frac{\gamma-1}{2}},\hat
j^a)\in W_{s,\delta+2}^p(\M)$, 
\begin{math}s\in (\frac{n}{p}, 
  \frac{2}{\gamma-1}+\frac{1}{p})\cap [1,\infty)\end{math} and 
$\delta\in(-\frac{n}{p},n-2-\frac{n}{p})$. 
\begin{enumerate} 
\item Assume $(\hat{z},\hat{j}^a)$ satisfy (\ref{eq:ee3}) with respect to a
flat metric $\hat{g}$, then
$(w,\tilde{u}^a)=\Phi^{-1}(z^{\frac{\gamma-1}{2}},\frac{j^a}{z})$
are initial data for the fluid and satisfy 
the compatibility condition (\ref{eq:ee7})
in the term of the metric $g=\phi^4\hat{g}$,
where  $z=\phi^{-8}\hat{z}$ and $j^a=\phi^{-10}\hat{j}^a$ and 
$\phi $ is the solution to the Lichnerowicz equation (\ref{eq:ee5}).
Moreover,
\begin{math} 
  (w,\tilde{u}^0-1, \tilde{u}^a)\in W_{s,\delta+2}^p(\M) 
\end{math}.

  \item There exists a conformal metric $g$, $(2,0)$--symmetric 
  form $K_{ab}$ which satisfy the constraint equation (\ref{eq:EE1}) with the
right hand side $(z,j^a)$. The pair $(M,g)$ is 
  asymptotically flat of the class $W_{s,\delta}^p$ and
  \begin{math} K_{ab}\in  W_{s-1,\delta+1}^p(\M)\end{math}.

\end{enumerate} 
 
\end{thm} 
 
\begin{rem} 
 The upper bound $\frac{2}{\gamma-1}+\frac{1}{p}$ for the regularity index $s$ 
is caused by the equation of state (\ref{eq:ee6}), and it is
not needed  whenever $\frac{2}{\gamma-1}$ is an integer.   
 
\end{rem}

\textit{Proof Theorem \ref{thm3}.} 
  We first  replace the metric $\bar g$ by a 
  conformal flat metric $\hat g$. The metric $\hat g$ is given by the 
  conformal transformation $\hat{g}=\varphi^4\bar g$, where 
  $\varphi-1\in W_{s,\delta}^p(\M)$.  The existence and the uniqueness 
  of such a $\varphi$ is assured by Theorem 
  \ref{thm:brill-cantor}. 
 
In the second stage we set $\hat{A}^{ab}=\varphi^{-10}\bar{A}^{ab}$ and  
\begin{equation*} 
 \hat{K}^{ab}=\hat{A}^{ab}+\left(\hat{\mathcal{L}}\left(W\right)\right)^{ab}, 
\end{equation*} 
where  
$\hat{\mathcal{L}}$ is is the Killing fields operator with respect to the
metric 
$\hat{g}$, that is,  
\begin{equation*} 
 \left({ 
      \hat{\mathcal{L}}}(W)\right)_{ab}= 
  \hat{\nabla}_a W_b+ \hat{\nabla}_b 
W_a-\frac{1}{3}g_{ab}\left(\hat{\nabla}_iW^i\right). 
\end{equation*} 
Then $\hat{K}$ satisfies the momentum constraint (\ref{eq:EE1}), if the  
 vector $W$ is a solution to the Lichnerowicz Laplacian  
\begin{equation} 
\label{eq:ee4} 
\left( 
\Delta_{L_{\hat{g}}}W\right)^b=\hat{\nabla}_a\left(\hat{\mathcal{L}} 
\left(W\right)\right)^{ab}=\Delta_{\hat{g}}W^b+\frac{1}{3}\hat{\nabla} 
^b\left(\hat{\nabla}_aW^a\right)+\hat{R}_a^b W^a=-8\pi \hat{j}^b. 
\end{equation}  
Here $\hat{R}_a^b$ is the Ricci curvature tensor with respect to the metric 
$\hat{g}$.  The Lichnerowicz Laplacian (\ref{eq:ee4}) is a strongly elliptic
operator (see 
e.g. \cite{y.00:_einst_euclid}) and belongs to ${\bf Asy}(\Delta,s,\delta,p)$, 
since $(\M,\hat{g})$ is asymptotically flat of the class $W_{s,\delta}^p$.
Its 
kernel 
consists of  Killing vector fields in $W_{s,\delta}^p(\M)$, since we assume 
there are no such fields, then  Corollary \ref{cor:1} implies that 
$\Delta_{L_{\hat{g}}}$ 
is an isomorphism, and consequently equation (\ref{eq:ee4}) possesses a 
solution. 
 
The solution to the Hamiltonian constraint is
constructed by an additional conformal
transformation $g=\phi^4\hat{g}$.
Setting $K^{ab}=\phi^{-10}\hat{K}^{ab}$ and $j^b=\phi^{-10}\hat{j}^b$
preserves the momentum constraint of (\ref{eq:EE1}) with respect to
the metric $g$.
Under this transformation, the scalar curvature $R(g)$ satisfies the
equation
\begin{equation*} 
\phi^{5} R(g)=R(\hat g)-8\Delta_{\hat g}\phi,
\end{equation*}

(see e.g. \cite[Ch. 5]{Aubin_1998}),   and since $R(\hat g)=0$, the Hamiltonian 
constraint in (\ref{eq:ee7}) is satisfied provided that $\phi$ is a solution to
the Lichnerowicz 
equation  
\begin{equation} 
\label{eq:ee5} 
 -\Delta_{\hat g}\phi =2\pi \hat z \phi^{-3}+\frac{1}{8}\hat{K}_a{}^b\hat{K}_b 
\phi^{-7}. 
\end{equation}  
Setting $u=\phi-1$, then the Lichnerowicz equation
(\ref{eq:ee5}) takes the form
\begin{equation*}
 -\Delta_{\hat g}u =2\pi \hat z (u+1)^{-3}+\frac{1}{8}\hat{K}_a{}^b\hat{K}_b{}^a 
(u+1)^{-7}, 
\end{equation*} 
which is then in a form suitable for the application  of Theorem \ref{thm2}.
This theorem provides a
non--negative solution $u\in W_{s,\delta}^p(\M)$. Hence $\phi\geq1$.

It remains to construct the initial data for the fluid variables 
$(w,\tilde{u}^\alpha)$ in terms of the  metric $g=\phi^4\hat{g}$.  
Setting  $z=\phi^{-8}\hat{z}$, preserves    the quantity
$\hat{z}^{-2}\hat{g}_{ab}\hat{j}^a\hat{j}^b$, while 
$z^{\frac{\gamma-1}{2}}=\phi^{-4(\gamma-1)}\hat{z}^{\frac{\gamma-1}{2}}$.   
Since  the adiabatic constant $\gamma>1$ and $\phi\geq1$, $\phi^{-4(\gamma-1)}\leq
1$ and consequently ${z}^{\frac{\gamma-1}{2}}\leq
\hat{z}^{\frac{\gamma-1}{2}}$. Therefore, if
$(\hat{z}^{\frac{\gamma-1}{2}},\frac{\hat j^a}{\hat{z}})$ satisfies
(\ref{eq:ee3}), then the 
pair $({z}^{\frac{\gamma-1}{2}}, \frac{\hat j^a}{z})$ does it too. 
 
 Hence, by Theorem \ref{thm4} we can let 
\begin{math} 
 (w,\tilde{u}^a)=\Phi^{-1}(z^{\frac{\gamma-1}{2}}, \frac{\hat j^a}{z})
\end{math}, and then obviously the compatibility conditions
(\ref{eq:ee7}) are satisfied in terms of the
metric $g$.
Notice that $z^{\frac{\gamma-1}{2}}\in W_{s,\delta+2}^p(\M)$, 
so now we can apply the estimate of the fractional power,
Proposition \ref{prop:5a}, with   $\beta=\frac{2}{\gamma-1}$ and obtain 
that $z\in W_{s,\delta+2}^p(\M)$.
At this stage appears the upper bound of the regularity index $s$.
From Propositions \ref{prop:4} and \ref{prop:Moser} we get that
$(w,\tilde{u}^a)=\Phi^{-1}(z^{\frac{\gamma-1}{2}}, \frac{\hat
  j^a}{z})$ are also in $W_{s,\delta+2}^p(\M)$.
Finally, since the velocity vector is a time--like unit vector, we set
$\tilde{u}^0=1+g_{ab}\tilde{u}^a\tilde{u}^b$.
\hfill{$\square$}

\section{Appendix}
\label{Appendix}

The tools and the elliptic theory  which we developed in $\setR^n$ hold also  
on an asymptotically flat manifolds of the class $W_{s,\delta}^p$.  Here we 
will discuss the extension of several properties to Riemannian manifolds. The 
unproven properties  are to be demonstrated by similar techniques and 
methods.  

We recall some of the notations of \S\ref{sec:Asymptotically flat manifold}:
From the Definition \ref{def:asymp}  there is a compact set 
$\mathcal{K}\subset \M$, and  a collection of charts $\{(U_i,\phi_i)\}_{i=1}^N$
such that $\M\setminus \mathcal{K}\subset \cup_{i=1}^N U_i$, where $\phi_i: 
E_{r_i}\to U_i$ is a homeomorphism and $E_{r_i}=\{x\in\setR^n: |x|>r_{r_i}\}$. 
Let  $\{(V_j,\varTheta_j)\}_{j=1}^{N_0}$ be a collection of charts 
that cover $\mathcal{K}$, where each 
$\varTheta_j$ is a homeomorphism between a ball  $B_j$  in $\setR^n$ and 
$V_j\subset \mathcal{K}$. Let  $\{\chi_i,\alpha_j\}$
be a partition of unity and  subordinate to $\{U_i,V_j\}$, then 
\begin{equation*}
\label{eq:norm-MA}
 \left\|u \right\|_{W_{s,\delta}^p(\M)}:= 
\sum_{i=1}^N\left\|\phi_i^\ast(\chi_i u)
\right\|_{W_{s,\delta}^p(\setR^n)}+\sum_{j=1}^{N_0}
\left\|\varTheta_j^\ast(\alpha_j u)
\right\|_{W_{s}^p(\setR^n)} \tag{\ref{eq:norm-M}}
\end{equation*} 
is a norm in $W_{s,\delta}^p(\M)$, and 
\begin{equation*}
\label{eq:bilinear:1A}
 \langle u, \varphi\rangle_{\M}=\sum_{i=1}^N\langle\phi_i^\ast(\chi_i u), 
\phi_i^\ast(\chi_i
\varphi)\rangle_W+\sum_{j=1}^{N_0}\langle\varTheta_j^\ast(\alpha_j
u),\varTheta_j^\ast(\alpha_j
\varphi)\rangle \tag{\ref{eq:bilinear:1}}
\end{equation*} 
is a bilinear form in $W_{s,\delta}^p(\M)$,  where $\langle \cdot,\cdot\rangle$
and $\langle \cdot,\cdot\rangle_W$ are the bilinear forms (\ref{eq:3}) and
(\ref{eq:4}) respectively.

We shall need the following elementary two Propositions.
\begin{prop}
\label{prop:appen:1}
Let $F:O_1\to O_2$ be a $C^\infty$--diffeomorphism between two open
set of $\setR^n$ such that $\det(DF)\geq \epsilon_0>0$. Then
\begin{equation}
\label{eq:appen:3}
\|F^\ast(u)\|_{W_{s,\delta}^p(O_1)}=\|u\circ F\|_{W_{s,\delta}^p(O_1)}\leq C
\|u\|_{W_{s,\delta}^p(O_2)},
\end{equation}
The constant $C$ depends on $s$, $\delta$ and $\epsilon_0$.
\end{prop}
\begin{proof}
 Let $s=k$ be a positive integer and consider the norm (\ref{eq:norm3}) 
restricted to the open sets $O_1$ and $O_2$. Then standard calculations give
that
\begin{equation*}
 \sum_{|\alpha|\leq 
k}\int_{O_1}\left(1+|x|\right)^{(\delta+|\alpha|)p}|\partial^\alpha (u\circ 
F)|^pdx \leq\left(\frac{C}{\epsilon_0}\right)^p
\sum_{|\alpha|\leq 
k}\int_{O_2}\left(1+|x|\right)^{(\delta+|\alpha|)p}|\partial^\alpha u|^pdx. 
\end{equation*}
Hence, by the equivalence of the norms (\ref{eq:intro:5}) and (\ref{eq:norm1})
(see Theorem \ref{thm:2}) we have that $\|F^\ast(u)\|_{W_{k,\delta}^p(O_1)}\leq 
\frac{C}{\epsilon_0}\|u\|_{W_{k,\delta}^p(O_2)}$.

For a negative integer $k$ we compute the norm in the dual form (\ref{eq:dual}).
Note that \\
\begin{math}
 (u\circ F,\varphi)_{L^2(O_1)}=(u,((\det(DF)^{-1}\varphi )\circ 
F^{-1})_{L^2(O_2)}
\end{math}, hence
by  Propositions \ref{prop:dual} and
\ref{prop:dual2}, and approximation, we have that
\begin{equation*}
\label{}
\begin{split}
 \|u\circ F\|_{W_{k,\delta}^{p}(O_1)} & =\sup\{|\langle u\circ
F,\varphi\rangle_W|: 
\|\varphi\|_{W_{-k,-\delta}^{p'}(O_1)}\leq 1, \varphi\in C_0^\infty(O_1) \}\\ 
 & =
\sup\{|\langle u,(\det(DF)^{-1}\varphi)\circ 
F^{-1}\rangle_W|: 
\|\varphi\|_{W_{-k,-\delta}^{p'}(O_1)}\leq 1, \varphi\in C_0^\infty(O_1) \}.
\end{split}
\end{equation*}
Replacing $F$ by $F^{-1}$, $p$ by $p'$ and $\delta$ by $-\delta$ in the
previous estimate, we obtain that
\begin{equation}
\begin{split}
|\langle u,\varphi (\det(DF)^{-1})\circ 
F^{-1})\rangle_W| & \leq C 
\|u\|_{W_{k,\delta}^p(O_2)}
\| (\det(DF)^{-1}\varphi)\circ 
F^{-1})\|_{W_{-k,-\delta}^{p'}(O_2)} \\  &\leq C
\|u\|_{W_{k,\delta}^p(O_2)}
\|\varphi\|_{W_{-k,-\delta}^{p'}(O_1)}.
\end{split}
\end{equation} 
Thus the linear operator (\ref{eq:appen:3}) is bounded whenever $s$ is an
integer. 
We now complete the proof by interpolation, Theorem \ref{thm1} (d). 
\end{proof}

\begin{rem}
 Obviously the proposition  holds also in the unweighted spaces $W_s^p$.
\end{rem}

The following proposition can be proven by  Proposition \ref{prop:appen:1} 
and by standard techniques of finite covering of manifolds (see 
e.g.~\cite{Driver_2003}).
\begin{prop}
\label{prop:appen:3}
 Suppose $u\in W_{s,\delta}^p(\M)$ and   $\supp(u)\subset U_i$ or
$\supp(u)\subset 
V_j$, then 
\begin{equation}
 \|u\|_{W_{s,\delta}^p(\M)}\leq C 
\|\phi_i^\ast(u)\|_{W_{s,\delta}^p(E_{r_i})}\quad \text{or} \qquad
\|u\|_{W_{s,\delta}^p(\M)}\leq C \|\varTheta_j^\ast(u)\|_{W_{s}^p(B_j)}.
\end{equation} 
\end{prop}

We are now in a position to extend several properties to the 
$W_{s,\delta}^p(\M)$--spaces. 
\begin{prop}
 \label{porp:approx_manofold}
 The class $C_0^\infty(\M)$ is dense in $W_{s,\delta}^p(\M)$.
\end{prop}
\begin{proof}
 Let $u\in W_{s,\delta}^p(\M)$, $\epsilon$ be a positive arbitrary number.  
Since 
$\phi_i^\ast\left(\chi_iu\right)$ has  support in $E_{r_i}$,  
 there is, by Theorem \ref{thm1} (b), $h_i\in C_0^\infty(\setR^n)$ such 
that 
$\left\|\phi_i^\ast\left(\chi_iu\right)-h_i\right\|_{W_{s,\delta}^p(E_{r_i})} 
<\epsilon $. Similarly,
$\varTheta_j^\ast(\alpha_j u)\in W_{s}^p(B_j)$ has compact support in the ball 
$B_j$, so by 
the approximation in the Besov 
spaces, there is $h_j\in C_0^\infty(B_j)$ such that 
$\left\|\varTheta_j^\ast(\alpha_j u)-h_j\right\|_{W_{s}^p(B_j)} < \epsilon$. 
Setting 
$h=\sum_{i=1}^N h_i\circ \phi_i^{-1}+\sum_{j=1}^{N_0}  h_j\circ \varTheta_j^{-1} 
$,  then $h\in C_0^\infty(\M)$ and from Proposition
\ref{prop:appen:3} we obtain that
\begin{equation*}
\begin{split}
 \|u-h\|_{W_{s,\delta}^p(\M)} &=\left\|\sum_{i=1}^N
\left(\chi_i u- h_i\circ
\phi_i^{-1}\right)+\sum_{j=1}^{N_0}\left(\alpha_j u -
 h_j\circ \varTheta_j^{-1}\right)\right\|_{W_{s,\delta}^p(\M)}\\ &
\leq \sum_{i=1}^N \left\|\left(\chi_i u- h_i\circ
\phi_i^{-1}\right)\right\|_{W_{s,\delta}^p(\M)}+\sum_{j=1}^{N_0}
\left\|\left(\alpha_j u -
 h_j\circ \varTheta_j^{-1}\right)\right\|_{W_{s,\delta}^p(\M)}
\\ &
\leq C \sum_{i=1}^N
\left\|\phi_i^\ast\left(\chi_i u\right)- h_i\right\|_{W_{s,\delta}^p(E_{r_i})}
+C\sum_{j=1}^{N_0}\left\|\varTheta_j^\ast\left(\alpha_j u\right) -
 h_j\right\|_{W_{s,\delta}^p(B_j)}
\\ 
&
\leq \epsilon C(N+N_0).
\end{split}
\end{equation*}

\end{proof}
The next proposition characterizes the topological dual of 
 $W_{s,\delta}^{p}(\M)$.
\begin{prop}
\label{prop:duality}
 Let $(\M,g)$ be an asymptotically flat manifold  of the class $W_{s,\delta}^p$,
then $W_{-s,-\delta}^{p'}(\M)= \left(W_{s,\delta}^p(\M)\right)^\ast$.
\end{prop}
\begin{proof}
From the bilinear (\ref{eq:bilinear:1}) form and the inequality 
(\ref{eq:bilinear:2}) we see 
that $W_{-s,-\delta}^{p'}(\M)\subset \left(W_{s,\delta}^p(\M)\right)^\ast$. 
Thus it suffices to show the reverse inclusion.

So let  $T\in  \left(W_{s,\delta}^p(\M)\right)^\ast$  and set
\begin{math}
 \|T\|= \sup\{ |T(\varphi)|: \|\varphi\|_{W_{s,\delta}^p(\M)}\leq
1\}
\end{math}.
For 
$i=1,\ldots,N$, we define 
$\phi^\ast_i(\chi_iT)(u)=T(\phi_i^\ast(\chi_i) u\circ\phi_i^{-1}))$, where 
$u\in W_{s,\delta}^p(\setR^n)$, and for  $j=1,\ldots,N_0$, we define 
$\varTheta^\ast_j(\alpha_jT)(u)=T(\varTheta_j^\ast(\alpha_j)
u\circ\varTheta_j^{-1}))$, where 
$u\in W_{s}^p(\setR^n)$.
Then by 
Proposition \ref{prop:appen:3}, 
\begin{equation}
\label{eq:appen:4}
\begin{split}
 |\phi^\ast_i(\chi_iT)(u)| &=|T(\phi_i^\ast(\chi_i) u\circ\phi_i^{-1}))|\leq 
\|T\|\|\phi_i^\ast(\chi_i) u\circ\phi_i^{-1})\|_{W_{s,\delta}^p(\M)}\\ &\leq C 
\|T\|\|\phi_i^\ast(\chi_i) u\|_{W_{s,\delta}^p(\setR^n)},
\end{split}
\end{equation}
and
\begin{equation}
\label{eq:appen:5}
\begin{split}
|\varTheta^\ast_j(\alpha_jT)(u)|& =|T(\varTheta_j^\ast(\alpha_j)
u\circ\varTheta_j^{-1}))|\leq 
\|T\|\|\varTheta_j^\ast(\alpha_j)
u\circ\alpha_j^{-1})\|_{W_{s,\delta}^p(\M)}\\ &\leq C 
\|T\|\|\varTheta_j^\ast(\alpha_j) u\|_{W_{s}^p(\setR^n)}.
\end{split}
\end{equation}

Thus $\phi^\ast_i(\chi_iT)\in \left(W_{s,\delta}^p(\setR^n)\right)^\ast$, and  
hence Theorem \ref{thm1} (c) implies that  $\phi^\ast_i(\chi_iT)\in 
W_{-s,-\delta}^{p'}(\setR^n)$. 
Similarly   $\varTheta^\ast_j(\alpha_jT)\in W_{-s}^{p'}(\setR^n)$.  Computing 
the norm of $T$ according to (\ref{eq:norm-M}), we obtain from
(\ref{eq:appen:4})  and (\ref{eq:appen:5}) that
\begin{equation*}
\begin{split}
& \|T\|_{W_{-s,-\delta}^{p'}(\M)}=
\sum_{i=1}^N\left\|\phi_i^\ast(\chi_i T)
\right\|_{W_{-s,-\delta}^{p'}(\setR^n)}+\sum_{j=1}^{N_0}
\left\|\varTheta^\ast_j(\alpha_j T)
\right\|_{W_{-s}^{p'}(\setR^n)}\\  = &
\sum_{i=1}^{N}\sup\{|\langle \phi^\ast_i(\chi_i T), u\rangle_W|:
\|u\|_{W_{s,\delta}^p(\setR^n)}\leq 1\}+\sum_{j=1}^{N_0}\sup\{|\langle 
\varTheta^\ast_j(\alpha_j T), u\rangle|:
\|u\|_{W_s^p(\setR^n)}\leq 1\}
\\ 
\leq & C \|T\|(N+N_0).
\end{split}
\end{equation*}
Hence $T\in W_{-s,-\delta}^{p'}(\M)$.

\end{proof}

We shall now extend the multiplicity property to an asymptotically flat manifold 
 of the class $W_{s,\delta}^p$.

 \begin{prop}[Proposition \ref{prop:4}]
 \label{prop:4M}
 Let $\M$ be to an asymptotically flat manifold 
 of the class $W_{s,\delta}^p$, and   assume $s\leq \min\{s_1,s_2\}$, 
$s_1+s_2>s+\frac{n}{p}$,  $s_1+s_2\geq
n\cdot\max\{0,(\frac{2}{p}-1)\}$ and
$\delta\leq\delta_1+\delta_2+ \frac{n}{p}$, then the multiplication 
\begin{equation*}
 W_{s_1,\delta_1}^p(\M) \times  W_{s_2,\delta_2}^p(\M)\to W_{s,\delta}^p(\M)
\end{equation*}
is continuous.
  
\end{prop}

\begin{proof}
Note that \begin{math}
\{\tilde{\chi}_i^2,\tilde{\alpha}_j^2\} :=
\left(\sum_{i=1}^N\chi_i^2+\sum_{j=1}^{N_0}\alpha_j^2\right)^{-1}\{\chi_i^2,
\alpha_j^2\}
\end{math} 
is also a partition of unity. Since different partitions of unity result in 
equivalent norms, we have that 
 \begin{equation*}
  \|uv\|_{W_{s,\delta}^p(\M)}\simeq\sum_{j=1}^{N_0} 
\|\varTheta_j^\ast\left(\alpha_j^2(uv)\right)\|_{W_{s}^p}+\sum_{i=0}^N
\|\phi_i^\ast(\chi_i^2 uv)\|_{W_{s,\delta}^p(\setR^n)}.
 \end{equation*}
From Proposition \ref{prop:4} we obtain that
\begin{equation*}
 \|\phi_i^\ast(\chi_i^2 
uv)\|_{W_{s,\delta}^p(\setR^n)}\lesssim \|\phi_i^\ast(\chi_i)
u)\|_{W_{s_1,\delta_1}^p(\setR^n)}\ \|\phi_i^\ast(\chi_i v)\|_{W_{s_2,\delta_2}
^p(\setR^n)},
\end{equation*}
and by the corresponding estimate in the unweighted Besov 
spaces (e.g. \cite[\S4.6.1]{runst96:_sobol_spaces_fract_order_nemyt}), we also 
have that 
\begin{equation*}
 \|\varTheta_j^\ast(\alpha_j^2
uv)\|_{W_{s}^p}\lesssim \|\varTheta_j^\ast(\alpha_j
u)\|_{W_{s_1}^p}\ \|\varTheta_j^\ast(\alpha_j 
v)\|_{W_{s_2}^p}.
\end{equation*}
Now we set $a_i=\|\phi_i^\ast(\chi_i) u)\|_{W_{s_1,\delta_1}^p(\setR^n)}$, $a_j= 
\|\varTheta_j^\ast(\alpha_ju)\|_{W_{s_1}^p}$, 
$b_i=\|\phi_i^\ast(\chi_i) v)\|_{W_{s_2,\delta_2}^p(\setR^n)}$ and $b_j = 
\|\varTheta_j^\ast(\alpha_j v)\|_{W_{s_2}^p}$.
Using the above two estimates and  an elementary 
inequality, we have that 
\begin{equation*}
 \begin{split}
\|uv\|_{W_{s,\delta}^p(\M)} & \simeq \sum_{i=1}^Na_ib_i
+\sum_{j=1}^{N_0} a_jb_j+ \leq  (N_0+N) \left(\sum_{i=1}^Na_i+ \sum_{j=1}^{N_0} 
a_j\right)
\left(\sum_{i=1}^Nb_i+ \sum_{j=1}^{N_0} b_j\right)
\\ &\lesssim \|u\|_{W_{s_1,\delta_1}^p(\M)} \|v\|_{W_{s_2,\delta_2}^p(\M)}.
\end{split}
\end{equation*}

\end{proof}

Finally we show Lemma \ref{lem:elliptic2} on
asymptotically  flat manifolds.
\begin{lem}[Lemma \ref{lem:elliptic2}]
\label{lemma:3.5}
{Let $L$ be an elliptic operator on asymptotically flat manifold   
$(\M,g)$ of class $W_{s,\delta}^p$ and assume 
$L\in{\bf Asy}(A_\infty,s,\delta,p)$,  $s\in(\frac{n}{p},\infty)\cap 
[1,\infty)$, $\delta \in 
(-\frac{n}{p},-2+\frac{n}{p^\prime})$
$p\in(1,\infty)$ and $\delta'<\delta$. Then for any $u\in
W_{s,\delta}^p(\M)$,
\begin{equation*}
 \left\|u\right\|_{W_{s,\delta}^p(\M)}\leq C\left\{
\left\|Lu\right\|_{W_{s-2,\delta+2}^p(\M)}+
\left\|u\right\|_{W_{s-1,\delta'}^p(\M)}\right\},
\end{equation*}
and the constant $C$ depends on  $W_{s,\delta}^p$-norms of the coefficients of
$L$, $s,\delta,p$ and $\delta'$.}
\end{lem}
\begin{proof}
Let $u\in W_{s,\delta}^p(\M)$ and set $\tilde{u}=\phi_i^\ast(u)$, 
$\tilde{\chi}=\phi_i^\ast(\chi_i)$  and
$\tilde{L}=\phi_i^\ast(L)$, then in local coordinates 
\begin{equation*}
(\tilde{L}\tilde{u})^i= 
(\tilde{a}_2)_{ij}^{ab}\left(\partial_a\partial_b\tilde{u}^j-\Gamma_{ab}
^c\partial_c \tilde{u}^j\right)+(\tilde{a}_1)_{ij} ^a\partial_a \tilde{u}^j 
+(\tilde{a}_0)_{ij}\tilde{u}^j,
\end{equation*}
where $\Gamma_{ab}^j$ 
denote the 
Christoffel symbols. Then 
$\tilde{L}\tilde{u}$ is an elliptic operator
in  the set $E_{r_i}$, and therefore it can be extend to $\setR^n$ 
such that it will remain elliptic in $\setR^n$. Using  
Definition \ref{def:asymp}, the multiplication  property, Proposition 
\ref{prop:4},   as well as
Propositions \ref{prop:1} and 
\ref{prop:Moser}, we obtain that $(\tilde{a}_2-A_\infty)\in 
W_{s,\delta}^p(\setR^n)$ and $(-\tilde{a}_2\Gamma_{ab}^c+\tilde{a}_1)\in 
W_{s-1,\delta+1}^p(\setR^n)$. 
 Hence we can apply Lemma  \ref{lem:elliptic2} and  obtain that 
\begin{equation}
\label{eq:appen:7.8}
 \left\|\tilde{\chi}^2\tilde{u}\right\|_{W_{s,\delta}^p(\setR^n)}\leq
C\left\{\left\|\tilde{L}(\tilde{\chi}^2\tilde{u})\right\|_{W_{s-2,\delta+2}
^p(\setR^n) }
+\left\|\tilde{\chi}^2\tilde{u}\right\|_{W_{s-1,\delta'}^p(\setR^n)}\right\}.
\end{equation} 
We can now write that
$\tilde{L}(\tilde{\chi}^2\tilde{u})=\tilde{\chi}^2\tilde{L}(\tilde{u})+
\tilde{\chi}R\tilde{u}$, where $R$ is an operator of first order and
its coefficients contain derivatives of $\tilde{\chi}$.
Since $\tilde{\chi}(x)=1$ for large $|x|$, the coefficients of $R$
have compact support, and therefore the norm
$\left\|\tilde{\chi}R\tilde{u}\right\|_{W_{s-2,\delta+2}^p(\setR^n)}$
is equivalent to
$\left\|\tilde{\chi}R\tilde{u}\right\|_{W_{s-2,\delta''}^p(\setR^n)}$
for any choice of $\delta''$.
Applying the multiplication properties,
Proposition \ref{prop:4}, we have that
\begin{equation*}
\begin{split}
\left\|\tilde{\chi}R\tilde{u}\right\|_{W_{s-2,\delta+2}^p(\setR^n)} & \leq 
C\left\{
\left\|\tilde{\chi}\partial\tilde{u}\right\|_{W_{s-2,\delta'+1}^p(\setR^n)}+
\left\|\tilde{\chi}\tilde{u}\right\|_{W_{s-2,\delta'}^p(\setR^n)}
\right\}\\ &\leq C
\left\{
\left\|\tilde{\chi}\tilde{u}\right\|_{W_{s-1,\delta'}^p(\setR^n)}+
\left\|\tilde{\chi}\tilde{u}\right\|_{W_{s-1,\delta'}^p(\setR^n)}\right\}
\leq C \left\|\tilde{\chi}\tilde{u}\right\|_{W_{s-1,\delta'}^p(\setR^n)}.
\end{split}
\end{equation*}
Since multiplication with $\tilde{\chi}$ or $\tilde{\chi}^2$ results
in equivalent norms, we conclude  from  the above inequality and 
(\ref{eq:appen:7.8})  
that 
\begin{equation}
\label{eq:appen:6}
 \left\|\tilde{\chi}\tilde{u}\right\|_{W_{s,\delta}^p(\setR^n)}\leq
C\left\{\left\|\tilde{\chi}\tilde{L}(\tilde{u})\right\|_{W_{s-2,
\delta+2 }
^p(\setR^n) }
+\left\|\tilde{\chi}\tilde{u}\right\|_{W_{s-1,\delta'}^p(\setR^n)}\right\}.
\end{equation}
For the covering of the compact part we apply Lemma 32 of 
\cite{Holst_Nagy_Tsogtgerel} and obtain that
\begin{equation} 
\label{eq:appen:7}
\left\|\varTheta_j^\ast(\alpha_j u)\right\|_{W_{s}^p(\setR^n)}
\leq
C\left\{\left\|\varTheta_j^\ast(\alpha_jL u)\right\|_{W_{s-2}
^p(\setR^n) } 
+\left\|\varTheta_j^\ast(\alpha_j 
u)\right\|_{W_{s-1}^p(\setR^n)}\right\}.
\end{equation}
Recalling the form of the norm (\ref{eq:norm-M}), we see that inequalities 
(\ref{eq:appen:6}) and (\ref{eq:appen:7}) complete the proof. 

\end{proof}

\vspace{5mm}

\bibliographystyle{amsplain}
\bibliography{bibgraf}

\end{document}